\documentclass[11pt,french]{smfart}
\usepackage{amsfonts, amssymb, amsmath, amsthm, latexsym, enumerate, epsfig, color, mathrsfs, fancyhdr, pdfsync, eurosym, endnotes, stmaryrd}
\usepackage[all]{xy}

\usepackage[frenchb,english]{babel}

\usepackage{mysections}
\usepackage{raccourcis}

\usepackage[mac]{inputenc}
\usepackage[T1]{fontenc}

\usepackage{hyperref}

\begin{document}
\setlength{\baselineskip}{0.55cm}	

\title{Les espaces de Berkovich sont ang\'eliques}
\author{J\'er\^ome Poineau}
\address{Institut de recherche math\'ematique avanc\'ee, 7, rue Ren\'e Descartes, 67084 Strasbourg, France}
\email{jerome.poineau@math.unistra.fr}
\thanks{L'auteur est membre du projet jeunes chercheurs \og Berko \fg\ de l'Agence Nationale de la Recherche.}

\date{\today}


\begin{abstract}
\selectlanguage{french}
Bien que les espaces de Berkovich d\'efinis sur un corps trop gros ne soient, en g\'en\'eral, pas m\'etrisables, nous montrons que leur topologie reste en grande partie gouvern\'ee par les suites : tout point adh\'erent \`a une partie est limite d'une suite de points de cette partie et les parties compactes sont s\'equentiellement compactes. Notre preuve utilise de fa\c{c}on essentielle l'extension des scalaires et nous en \'etudions certaines propri\'et\'es. Nous montrons qu'un point d'un disque peut \^etre d\'efini sur un sous-corps de type d\'enombrable et que, lorsque le corps de base est alg\'ebriquement clos, tout point est universel : dans une extension des scalaires, il se rel\`eve canoniquement.

\vskip 0.5\baselineskip

\selectlanguage{english}
\noindent{\bf Abstract}
\vskip 0.5\baselineskip
\noindent
{\bf Berkovich spaces are angelic.} Although Berkovich spaces may fail to be metrizable when defined over too big a field, we prove that a large part of their topology can be recovered through sequences: for instance, limit points of subsets are actual limits of sequences and compact subsets are sequentially compact. Our proof uses extension of scalars in an essential way and we need to investigate some of its properties. We show that a point in a disc may be defined over a subfield of countable type and that, over algebraically closed fields, every point is universal: in an extension of scalars, it may be canonically lifted.
\end{abstract}

\maketitle

\newpage

\section{Introduction}\label{section:intro}

Parmi les diff\'erentes th\'eories d'espaces analytiques $p$-adiques, celle introduite par V.~Berkovich se distingue entre autres par ses propri\'et\'es topologiques agr\'eables. Mentionnons, par exemple, qu'en d\'epit du caract\`ere totalement discontinu du corps de base, les espaces de Berkovich sont localement compacts, localement connexes par arcs (\emph{cf.}~\cite{rouge}) et m\^eme localement contractiles dans de nombreux cas (\emph{cf.}~\cite{smoothI}, \cite{HL}). 

Pour des applications \`a la th\'eorie des syst\`emes dynamiques, ces propri\'et\'es pr\'esentent un grand int\'er\^et et ont d\'ej\`a rendu de nombreux services (\emph{cf.}~\cite{BR} pour un expos\'e d\'etaill\'e dans le cadre de la droite). Cependant, d'autres obstacles se pr\'esentent : dans ce contexte, les suites jouent un r\^ole pr\'epond\'erant, mais leur comportement ne pr\'esente {\it a priori} gu\`ere de liens avec la topologie des espaces de Berkovich. En effet, lorsque leur corps de d\'efinition est trop gros, ces espaces cessent en g\'en\'eral d'\^etre m\'etrisables et rien n'assure alors que les caract\'erisations usuelles des propri\'et\'es topologiques en termes de suites continuent de s'appliquer. Pourtant, nous allons montrer que, dans une large mesure, tel est bien le cas, allongeant ainsi la liste des propri\'et\'es topologiques remarquables des espaces de Berkovich.

Dans ce texte, nous allons pr\'ecis\'ement montrer que les espaces de Berkovich sont des espaces de Fr\'echet-Urysohn. Cette condition, qui signifie que tout point adh\'erent \`a une partie est limite d'une suite de points de cette partie, entra\^{\i}ne notamment que la notion de partie ouverte ou ferm\'ee peut se tester \`a l'aide de suites et que toute partie compacte (au sens o\`u tout recouvrement ouvert poss\`ede un sous-recouvrement fini) est \'egalement s\'equentiellement compacte (au sens o\`u toute suite poss\`ede une sous-suite convergente). Nous d\'emontrerons \'egalement que les espaces de Berkovich sont ang\'eliques, c'est-\`a-dire qu'ils satisfont la condition suppl\'ementaire que leurs parties relativement $\omega$-compactes sont relativement compactes, sous certaines conditions, toujours v\'erifi\'ees pour les courbes ou les espaces provenant de vari\'et\'es alg\'ebriques.

Signalons que ces r\'esultats ont \'et\'e obtenus par C.~Favre dans~\cite{Countability} lorsque le corps de base est un corps de s\'eries de Laurent. Nous nous affranchissons ici de cette hypoth\`ese. Indiquons que la strat\'egie qu'il emploie fait intervenir des espaces de Riemann-Zariski (ce qui explique la restriction aux corps de s\'eries de Laurent) et se distingue assez nettement de la n\^otre.\\

Ajoutons, \`a pr\'esent, quelques mots sur les r\'esultats topologiques. Dans des cas simples comme celui du disque de dimension~$1$, voire celui des courbes, l'on se convainc assez facilement de leur v\'eracit\'e. Consid\'erons par exemple le disque unit\'e~$\D^{1,\mathrm{an}}_{k}$ sur un corps valu\'e complet alg\'ebriquement clos~$k$ de valuation non triviale. Les $k$-points y sont denses et l'on construit ais\'ement une suite de $k$-points convergeant vers un point donn\'e {\it a priori}. Si ce point est le point de Gau{\ss}, par exemple, il suffit que les points de la suite parcourent une infinit\'e de branches issues de ce point (autrement dit que l'ensemble des classes r\'esiduelles dans~$\ti{k}$ des points de la suite soit infini). Un raisonnement du m\^eme style montrerait que~$\D^{1,\mathrm{an}}_{k}$ est s\'equentiellement compact.

Cependant, la port\'ee de ce type de techniques semble assez limit\'ee et nous proc\'ederons par d'autres m\'ethodes. Les espaces de Berkovich \'etant construits \`a partir d'espaces affino\"{\i}des, qui sont eux-m\^emes des ferm\'es de Zariski de disques, il suffit en r\'ealit\'e de d\'emontrer les r\'esultats pour les disques. Exposons en quelques mots la strat\'egie que nous adopterons. \'Etant donn\'e un disque~$\D^{n,\mathrm{an}}_{k}$ sur un corps valu\'e complet~$k$, nous commencerons par descendre le probl\`eme sur un sous-corps~$\ell$ de~$k$. Si ce corps est assez petit, le disque~$\D^{n,\mathrm{an}}_{\ell}$ est m\'etrisable et le probl\`eme est r\'esolu. Il faut alors remonter.

Nous consacrons la section~\ref{section:universel} aux points \og que l'on peut remonter \fg. Ces points, que nous appelons universels, ont \'et\'e introduits par V.~Berkovich sous le nom de \og peaked points \fg. Sans rentrer dans les d\'etails, un point~$x$ d'un espace analytique~$X$ sur~$k$ est dit universel lorsque, pour toute extension valu\'ee compl\`ete~$K$ de~$k$, il existe un point canonique au-dessus de~$x$ dans~$X_{K}$. Examinons, de nouveau, le cas o\`u~$X$ est un disque de dimension~$1$ sur un corps alg\'ebriquement clos. On peut alors \'ecrire tout point comme bord de Shilov (qui co\"{\i}ncide ici avec le bord topologique) d'un disque centr\'e en un point rationnel de rayon plus petit ou comme limite de tels bords. Nous obtenons ainsi un proc\'ed\'e pour remonter canoniquement tous les points de~$X$ \`a~$X_{K}$. C'est l'approche adopt\'ee par X.~Faber dans~\cite{Berkovichramification}, \S 4.

De fa\c{c}on plus g\'en\'erale, V.~Berkovich a propos\'e un moyen de v\'erifier la propri\'et\'e d'universalit\'e consistant \`a r\'ealiser le point comme unique point du bord de Shilov d'un espace strictement affino\"{\i}de d'un certain type. Nous aurons besoin de g\'en\'eraliser ce r\'esultat \`a des espaces qui ne sont pas strictement affino\"{\i}des, ce qui oblige \`a remplacer les r\'eductions classiques par des r\'eductions gradu\'ees au sens de M.~Temkin (\emph{cf.}~\cite{TemkinII}). C'est pour cette raison que nous avons r\'edig\'e la section~\ref{section:alggrad}, o\`u nous \'etendons au cadre gradu\'e quelques r\'esultats classiques d'alg\`ebre commutative tels le Nullstellensatz. Nous en d\'eduisons le r\'esultat de \og remont\'ee \fg\ que nous recherchions : sur un corps alg\'ebriquement clos, tout point est universel.

\`A la section~\ref{section:universeldisque}, nous nous int\'eressons sp\'ecifiquement au cas des disques. Ainsi que nous l'avons expliqu\'e pr\'ec\'edemment, nous montrons que nous pouvons en \og descendre \fg\ les points : tout point du disque~$\D^{n,\mathrm{an}}_{k}$ peut \^etre d\'efini (en un sens que nous pr\'ecisons) sur un disque~$\D^{n,\mathrm{an}}_{\ell}$, o\`u~$\ell$ est un sous-corps de~$k$ de type d\'enombrable sur son sous-corps premier. Dans ce cas, le disque~$\D^{n,\mathrm{an}}_{\ell}$ est m\'etrisable. 

Nous disposons alors de tous les ingr\'edients pour mettre en place la strat\'egie expos\'ee plus haut et en tirer quelques cons\'equences ayant trait \`a la topologie des espaces de Berkovich. C'est l'objet de la section~\ref{section:topo}.\\

\textbf{Remerciements} 

Nous remercions tr\`es chaleureusement Charles Favre qui nous a pos\'e les questions \`a l'origine de ce texte et nous a communiqu\'e sa pr\'epublication~\cite{Countability} sur le m\^eme sujet. Ses remarques sur le texte nous ont apport\'e une aide certaine et nous ont amen\'e \`a \'eclaircir plusieurs points. Merci \'egalement \`a Antoine Chambert-Loir et Antoine Ducros pour leurs commentaires et conseils.\\

\textbf{Notations}

Ce texte est consacr\'e \`a l'\'etude d'espaces analytiques au sens de V.~Berkovich. Nous adopterons les notations suivantes dans tout le texte, exception faite de la section~\ref{section:alggrad}. La lettre~$k$ d\'esignera un corps muni d'une valeur absolue ultram\'etrique pour laquelle il est complet. Nous n'excluons pas le cas de la valeur absolue triviale. Nous noterons~$k_{p}$ le compl\'et\'e du sous-corps premier de~$k$. 

Si~$X$ est un espace $k$-analytique et~$K$ une extension valu\'ee compl\`ete de~$k$, nous noterons~$X_{K}$ l'espace $K$-analytique obtenu par extension des scalaires.

Nous dirons qu'une famille finie~$\br=(r_{1},\dots,r_{n})$ de nombres r\'eels strictement positifs est un polyrayon $k$-libre si son image dans le~$\Q$-espace vectoriel~$\R_{+}^*/\sqrt{|k^*|}$ est une famille libre. Nous noterons~$k_{\br}$ le corps constitu\'e des s\'eries de la forme 
\[f = \sum_{\bm\in\Z^n} \alpha_{\bm}\, T_{1}^{m_{1}}\dots T_{n}^{m_{n}},\] 
avec $\alpha_{\bm} \in k$, telles que la famille $(|\alpha_{\bm}| r_{1}^{m_{1}}\dots r_{n}^{m_{n}})_{\bm\in\Z^n}$ est sommable. Muni de la norme d\'efinie par $\|f\|_{\br} = \max_{\bm\in\Z^n}(|\alpha_{\bm}| r_{1}^{m_{1}}\dots r_{n}^{m_{n}})$, c'est un corps valu\'e complet. 

Soit~$(A,\|.\|)$ une $k$-alg\`ebre de Banach. L'alg\`ebre~$A\ho_{k} k_{\br}$ (\emph{cf.}~section~\ref{section:universel} pour des rappels sur le produit tensoriel compl\'et\'e) est alors isomorphe \`a l'espace des s\'eries de la forme $\sum_{\bm\in\Z^n} a_{\bm}\, T_{1}^{m_{1}}\dots T_{n}^{m_{n}}$, avec $a_{\bm}\in A$, telles que la famille $(\|a_{\bm}\| r_{1}^{m_{1}}\dots r_{n}^{m_{n}})_{\bm\in\Z^n}$ est sommable. Muni de la norme d\'efinie par $\|f\|_{\br} = \max_{\bm\in\Z^n}(\|a_{\bm}\| r_{1}^{m_{1}}\dots r_{n}^{m_{n}})$, c'est un espace de Banach.

Soient~$\As$ une alg\`ebre $k$-affino\"{\i}de. Tout point~$x$ du spectre analytique \mbox{$X=\Mc(\As)$} de~$\As$ est associ\'e \`a une semi-norme multiplicative born\'ee sur~$\As$, que nous noterons~$|.|_{x}$. Nous noterons~$\p_{x}$ l'id\'eal premier de~$\As$ d\'efini par 
\[\p_{x} = \{f\in\As\, |\, |f|_{x} = 0\},\]
$k(\p_{x})$ le corps des fractions de l'anneau int\`egre~$\As/\p_{x}$ et~$\Hs(x)$ son compl\'et\'e pour la valeur absolue induite par~$|.|_{x}$.

Suivant~\cite{rouge}, \S~9.1, si~$K$ est une extension valu\'ee compl\`ete de~$k$, nous noterons~$s(K/k)$ le degr\'e de transcendance du corps~$\ti{K}$ sur~$\ti{k}$ et~$t(K/k)$ la dimension du $\Q$-espace vectoriel~$\sqrt{|K^*|}/\sqrt{|k^*|}$. Pour un point~$x$ d'un espace $k$-affino\"{\i}de~$X$, nous noterons $s_{k}(x)=s(\Hs(x)/k)$ et $t_{k}(x)=t(\Hs(x)/k)$ (ou simplement~$s(x)$ et~$t(x)$ si aucune confusion ne peut en r\'esulter). Dans ce contexte, l'in\'egalit\'e d'Abhyankar s'\'ecrit
\[s_{k}(x)+t_{k}(x) \le \dim_{k,x}(X).\]
Pour des rappels sur la dimension des espaces analytiques, nous renvoyons \`a~\cite{variationdimension}, \S~1.

\section{Quelques r\'esultats d'alg\`ebre gradu\'ee}\label{section:alggrad}

Dans cette section, nous d\'emontrons quelques r\'esultats d'alg\`ebre gradu\'ee, au sens de M.~Temkin. Fixons un groupe commutatif~$G$. Les anneaux gradu\'es que nous consid\'erons sont des anneaux commutatifs et unitaires~$A$, munis de $G$-graduations $A = \bigoplus_{g\in G} A_{g}$ par des sous-groupes additifs v\'erifiant $A_{g}A_{h} \subset A_{gh}$. Nous demandons que les morphismes respectent la graduation. Les notions d'alg\`ebre commutative habituelles s'\'etendent \`a ce cadre. Nous renvoyons \`a~\cite{TemkinII}, \S 1 pour les d\'etails mais rappelons tout de m\^eme quelques d\'efinitions pour la commodit\'e du lecteur.

\begin{itemize}
\item Un \'el\'ement homog\`ene~$x$ de~$A$ est un \'el\'ement appartenant \`a l'un des~$A_{g}$. On note appelle~$g$ l'ordre de~$x$ et on le note~$\rho(x)$.
\item Un anneau gradu\'e est dit int\`egre si le produit de deux \'el\'ements homog\`enes non nuls est non nul.
\item Un corps gradu\'e est un anneau gradu\'e non nul dans lequel tous les \'el\'ements homog\`enes sont inversibles. Tout anneau gradu\'e int\`egre poss\`ede un corps des fractions gradu\'e obtenu en inversant les \'el\'ements homog\`enes.
\item Un id\'eal homog\`ene~$I$ de~$A$ est un id\'eal engendr\'e par des \'el\'ements homog\`enes. Le quotient~$A/I$ est alors encore naturellement un anneau gradu\'e.
\item Si~$g$ est un \'el\'ement de~$G$, l'anneau $B = A[g^{-1}T]$ est l'anneau~$A[T]$ muni de l'unique graduation telle que~$T$ soit homog\`ene d'ordre~$g$.
\item Une $A$-alg\`ebre gradu\'ee~$B$ est de type fini si elle est quotient d'une alg\`ebre gradu\'ee~$A[\bg^{-1}\bT]$ par un id\'eal homog\`ene.
\item Un $A$-module~$B$ est de type fini s'il est quotient d'un module gradu\'e~$A^n$ (avec la graduation donn\'ee par $(A^n)_{g} = A_{g}^n$) par un sous-module homog\`ene.
\end{itemize}

Signalons que certaines de ces notions peuvent se comporter de fa\c{c}on surprenante. Consi\-d\'erons, par exemple, un corps gradu\'e~$k$, un \'el\'ement~$g$ de~$G$ n'appartenant pas \`a~$\rho(k)$ et consid\'erons l'anneau gradu\'e~$k[g^{-1}T]$. Ses seuls \'el\'ements homog\`enes sont de la forme $aT^n$, o\`u~$a$ est un \'el\'ement homog\`ene de~$k$ et~$n$ un entier. Ce sont donc les seuls \'el\'ements que l'on inverse pour obtenir le corps des fractions gradu\'e~$k(g^{-1}T)$. Cet argument montre que le corps gradu\'e~$k(g^{-1}T)$ est de type fini sur~$k$. En particulier, nous ne pouvons esp\'erer dans ce cadre un Nullstellensatz compl\`etement analogue au Nullstellensatz classique.\\

Ce formalisme des anneaux gradu\'es nous sera utile par la suite pour r\'eduire des alg\`ebres $k$-affino\"{\i}des qui ne sont pas n\'ecessairement $k$-affino\"{\i}des. En g\'en\'eral, si~$\Ds$ d\'esigne une $k$-alg\`ebre de Banach et $\rho : \Ds \to \R_{+}$ sa semi-norme spectrale, pour tout nombre r\'eel $r>0$, on d\'efinit~$\ti{\Ds}_{r}$ comme le quotient de l'anneau $\{x\in \Ds\, |\, \rho(x)\le r\}$ par l'id\'eal $\{x\in \Ds\, |\, \rho(x) < r\}$. La r\'eduction de~$\Ds$ est alors l'alg\`ebre $\R_{+}^*$-gradu\'ee
\[\ti{\Ds} = \bigoplus_{r>0} \ti{\Ds}_{r}.\]
Cette r\'eduction appliqu\'ee \`a des alg\`ebres $k$-affino\"{\i}des quelconques poss\`ede des propri\'et\'es similaires \`a la r\'eduction classique des alg\`ebres strictement $k$-affino\"{\i}des. Pour des r\'esultats pr\'ecis, nous renvoyons \`a~\cite{TemkinII}, \S 3.\\

Notre but est ici de d\'emontrer une version du Nullstellensatz pour les alg\`ebres gradu\'ees. Nous suivons la preuve qu'en ont propos\'ee E.~Artin et J.~Tate dans~\cite{ArtinTate}. Nous en profitons pour donner quelques analogues de d\'efinitions et r\'esultats classiques. Commen\c{c}ons par \'enoncer deux r\'esultats concernant les alg\`ebres gradu\'ees de polyn\^omes en une variable. Nous en omettons les d\'emonstrations, en tout point analogues aux d\'emonstrations classiques.

\begin{prop}[Division euclidienne]
Soient~$k$ un corps gradu\'e et~$g$ un \'el\'ement de~$G$. Soient~$A$ et~$B$ deux \'el\'ements homog\`enes de l'anneau gradu\'e~$k[g^{-1}T]$. Supposons que~$B$ n'est pas nul. Alors il existe un unique couple~$(Q,R)$ d'\'el\'ements homog\`enes de~$k[g^{-1}T]$ qui v\'erifient les propri\'et\'es suivantes :
\begin{enumerate}[\it i)]
\item $A = BQ +R$\ ;
\item $\deg(R) < \deg(B)$.
\end{enumerate}
\end{prop}

\begin{cor}\label{cor:principal}
Soit~$k$ un corps gradu\'e. Pour tout \'el\'ement~$g$ de~$G$, l'anneau gradu\'e $k[g^{-1}T]$ est principal : tout id\'eal homog\`ene est engendr\'e par un \'el\'ement homog\`ene.
\end{cor}

Comme dans le cas classique, on peut d\'efinir sur les corps gradu\'es une notion d'\'el\'ement alg\'ebrique.

\begin{defi}
Soit~$\bg$ un \'el\'ement de~$G^n$. Un polyn\^ome~$P(\bT)$ en~$n$ variables \`a coefficients dans une alg\`ebre gradu\'ee~$A$ est dit {\bf $\bg$-homog\`ene} s'il d\'efinit un \'el\'ement homog\`ene de~$k[\bg^{-1}\bT]$. Un polyn\^ome est dit {\bf $G$-homog\`ene}, ou simplement {\bf homog\`ene}, s'il existe un \'el\'ement~$\bg$ de~$G^n$ pour lequel il est $\bg$-homog\`ene.
\end{defi}

\begin{defi}
Soit~$k$ un corps gradu\'e. Un \'el\'ement homog\`ene~$x$ d'une $k$-alg\`ebre gradu\'ee est dit {\bf alg\'ebrique} sur~$k$ s'il est racine d'un polyn\^ome $G$-homog\`ene en une variable \`a coefficients dans~$k$.

Le corps gradu\'e~$k$ est dit {\bf alg\'ebriquement clos} si tout polyn\^ome $G$-homog\`ene non constant en une variable \`a coefficients dans~$k$ y poss\`ede une racine.
\end{defi}

\begin{rem}
Le corollaire~\ref{cor:principal} permet notamment de d\'efinir le polyn\^ome minimal d'un \'el\'ement alg\'ebrique sur~$k$. C'est un polyn\^ome $G$-homog\`ene.
\end{rem}

On v\'erifie facilement qu'un corps gradu\'e est alg\'ebriquement clos si, et seulement si, toutes ses extensions finies sont triviales. Comme dans le cas classique, les r\'eductions des corps complets alg\'ebriquement clos sont alg\'ebriquement closes, ainsi que l'exprime la proposition suivante. 

\begin{prop}
Soit~$k$ un corps valu\'e complet alg\'ebriquement clos. Alors sa r\'eduction~$\ti{k}$ est un corps gradu\'e alg\'ebriquement clos.
\end{prop}
\begin{proof}
Ici la graduation est celle qui correspond \`a la valeur absolue et le groupe~$G$ est n'est autre que le groupe multiplicatif~$\R_{+}^*$.

Soient $r>0$ et $\ti{P}(T) = \sum_{n=0}^d \alpha_{n} T^n$ un \'el\'ement homog\`ene non constant de~$\ti{k}[r^{-1}T]$. Notons~$s$ son ordre. Nous pouvons supposer que~$\alpha_{0} \ne 0$. Relevons~$\ti{P}(T)$ en un \'el\'ement $P(T) = \sum_{n=0}^d a_{n} T^n$ de~$k[T]$. Puisque~$\ti{P}(T)$ est homog\`ene d'ordre~$s$, pour tout~$n$, nous avons $|a_{n}| r^n =s$.

Le polyn\^ome~$P(T)$ n'est pas constant et poss\`ede donc une racine~$x\ne 0$ dans~$k$. Il existe donc deux entiers $i,j\in\cn{0}{d}$, avec $i<j$, tels que l'on ait $|a_{i}x^i| = |a_{j}x^j| >0$. On en d\'eduit que $|x|^{j-i} = |a_{i}|/|a_{j}| = r^{j-i}$ et donc que $|x|=r$. Par cons\'equent, pour tout~$n$, $|a_{n}x^n| = s$ et $\ti{a}_{n} \ti{x}^n$ est homog\`ene de degr\'e~$s$. On en d\'eduit que $\sum_{n=0}^d \ti{a}_{n} \ti{x}^n = 0$ dans~$\ti{k}$.
\end{proof}

Nous disposons \'egalement d'une notion de module gradu\'e noeth\'erien et d'anneau noeth\'erien qui se comporte comme dans le cadre classique. De nouveau, nous laissons au lecteur le soin d'effectuer ces v\'erifications.

\begin{thm}[Artin-Tate]
Soient $A \subset B \subset C$ des anneaux gradu\'es. Supposons que~$A$ est noeth\'erien, que~$C$ soit une $A$-alg\`ebre de type fini et un $B$-module de type fini. Alors~$B$ est une $A$-alg\`ebre de type fini.
\end{thm}
\begin{proof}
Soient~$c_{1},\ldots,c_{n}$ des \'el\'ements homog\`enes de~$C$ qui l'engendrent en tant que $A$-alg\`ebre. Soient~$d_{1},\ldots,d_{m}$ des \'el\'ements homog\`enes de~$C$ qui l'engendrent en tant que $B$-module. Pour tout~$i$, nous pouvons \'ecrire
\[c_{i} = \sum_{j} \gamma_{i,j} d_{j},\]
o\`u les~$\gamma_{i,j}$ sont des \'el\'ements homog\`enes de~$A$ et, pour tous~$i$ et~$j$, nous pouvons \'ecrire
\[d_{i}d_{j} = \sum_{k} \delta_{i,j,k} d_{k},\]
o\`u les~$\delta_{i,j,k}$ sont des \'el\'ements homog\`enes de~$B$. Soit~$B_{0}$ la sous-alg\`ebre gradu\'ee de~$B$ engendr\'ee par les~$\gamma_{i,j}$ et les~$\delta_{i,j,k}$. Elle est de type fini sur~$A$ et donc noeth\'erienne. 

En utilisant le fait que tout \'el\'ement homog\`ene de~$C$ peut s'\'ecrire comme un polyn\^ome homog\`ene en les~$c_{i}$ \`a coefficients dans~$A$, on montre que~$C$ est une $B_{0}$-alg\`ebre gradu\'ee de type fini. On en d\'eduit que~$B$ est une $B_{0}$-alg\`ebre gradu\'ee de type fini et donc une $A$-alg\`ebre gradu\'ee de type fini.
\end{proof}

Nous avons d\'ej\`a remarqu\'e plus haut que certains corps gradu\'es transcendants peuvent \^etre de type fini. Nous apportons ici quelques pr\'ecisions sur ce r\'esultat. Si~$\bg$ d\'esigne un \'el\'em\'ent de~$G^n$, nous noterons~$k(\bg^{-1}\bT)$ le corps des fractions gradu\'e de l'anneau gradu\'e~$k[\bg^{-1}\bT]$.

\begin{defi}
Soit~$E$ une partie de~$G$. Une famille~$\Fs$ d'\'el\'ements de~$G$ est dite {\bf ind\'ependante de~$E$} si son image dans le $\Z$-module $G/ \la E \ra$ est libre.

Si la partie~$E$ est r\'eduite \`a l'\'el\'ement~$1$, nous dirons simplement que la famille est {\bf ind\'ependante}.
\end{defi}

\begin{lem}\label{lem:corpstranscendant}
Soient~$k$ un corps gradu\'e, $n$ un entier et~$\bg$ un \'el\'ement de~$G^n$. Le corps gradu\'e~$k(\bg^{-1}\bT)$ est une $k$-alg\`ebre gradu\'ee de type fini si, et seulement si, la famille~$\bg$ est ind\'ependante de~$\rho(k)$.
\end{lem}
\begin{proof}
Si $\bg=(g_{1},\ldots,g_{n})$ est ind\'ependante de~$\rho(k)$, nous avons
\[k(\bg^{-1}\bT) = k[g_{1}^{-1}T_{1},g_{1}S_{1},\ldots, g_{n}^{-1}T_{n},g_{n}S_{n}]/(T_{1}S_{1}-1,\ldots,T_{n}S_{n}-1)\]
et cette alg\`ebre gradu\'ee est donc de type fini sur~$k$.

Pour prouver la r\'eciproque, il suffit de d\'emontrer que si~$g$ est un \'el\'ement de~$G$ d\'ependant de~$\rho(k)$, alors le corps gradu\'e~$k(g^{-1} T)$ n'est pas de type fini. Consid\'erons donc un tel \'el\'ement~$g$ de~$G$  et supposons, par l'absurde, qu'il existe des \'el\'ements homog\`enes $P_{1}(T),\ldots,P_{n}(T)$ de~$k[g^{-1}T]$ d'ordres respectifs $h_{1},\ldots,h_{n}$ tels que 
\[k(g^{-1}T) = k[g^{-1}T,h_{1}P_{1}(T)^{-1},\ldots,h_{n}P_{n}(T)^{-1}].\]
Puisque~$g$ est d\'ependant de~$\rho(k)$, il existe un entier non nul~$d$ et un \'element non nul~$\alpha$ de~$k$ tels que $\prod_{i=1}^n h_{i}^d = \rho(\alpha)$. Le polyn\^ome $Q(T) = \prod_{i=1}^n P_{i}(T)^d - \alpha$ est donc un \'el\'ement homog\`ene de~$k[g^{-1}T]$. Si son inverse appartenait \`a l'alg\`ebre gradu\'ee $k[g^{-1}T,h_{1}P_{1}(T)^{-1},\ldots,h_{n}P_{n}(T)^{-1}]$, le polyn\^ome~$Q$ serait multiple de l'un (au moins) des polyn\^omes~$P_{i}$. Par division euclidienne, ceci est impossible.
\end{proof}

Signalons que les \'el\'ements alg\'ebriques sur les corps du type~$k(\bg^{-1}\bT)$,  o\`u~$\bg$ est une famille ind\'ependante de~$\rho(k)$, sont faciles \`a d\'ecrire.

\begin{lem}
Soient~$k$ un corps gradu\'e et~$\bg$ une famille d'\'el\'ements de~$G$ ind\'ependante de~$\rho(k)$. Tout \'el\'ement alg\'ebrique sur~$k(\bg^{-1}\bT)$ est le produit d'un \'el\'ement alg\'ebrique sur~$k$ par une racine d'un polyn\^ome de Kummer de la forme $X^n - \bT^\bm$, avec $n \ge 1$ et $\bm\in\Z^n$. 

En particulier, si le corps gradu\'e~$k$ est alg\'ebriquement clos, toute extension finie~$k(\bg^{-1}\bT)$ se plonge dans un corps gradu\'e de la forme~$k(\bh^{-1}\bS)$, o\`u~$\bh$ est une famille finie ind\'ependante de~$\rho(k)$.
\end{lem}
\begin{proof}
Soit~$x$ un \'el\'ement alg\'ebrique sur~$k(\bg^{-1}\bT)$. Consid\'erons son polyn\^ome minimal unitaire~$P(X)$, qui est un polyn\^ome homog\`ene. Son degr\'e~$d$ est un entier non nul et son coefficient constant est un \'el\'ement homog\`ene non nul de~$k(\bg^{-1}\bT)$, donc de la forme~$\alpha\bT^\bm$ avec $\alpha\in k^*$ et $\bm\in\Z^n$. Soit~$y$ une racine du polyn\^ome de Kummer~$X^d - \bT^\bm$. En divisant le polyn\^ome~$P(X)$ par~$y^d$, on montre que~$x/y$ est racine d'un polyn\^ome homog\`ene unitaire dont le coefficient constant est \'egal \`a~$\alpha$. En utilisant l'homog\'en\'eit\'e du polyn\^ome et le fait que~$\bg$ est ind\'ependante de~$\rho(k)$, on montre que tous ses coefficients appartiennent \`a~$k$. Autrement dit, $x/y$ est alg\'ebrique sur~$k$.
\end{proof}

D\'emontrons maintenant l'analogue du Nullstellensatz pour les alg\`ebres gradu\'ees.

\begin{cor}[Nullstellensatz gradu\'e]
Soit~$k$ un corps gradu\'e. Soit~$K$ une $k$-alg\`ebre gradu\'ee de type fini qui est un corps. Alors il existe une famille finie~$\bT$ d'\'el\'ements homog\`enes de~$K$ dont la famille des ordres~$\bg$ est ind\'ependante de~$\rho(k)$ telle que~$K$ soit une extension finie de~$k(\bg^{-1}\bT)$. 

En particulier, si le corps gradu\'e~$k$ est alg\'ebriquement clos, $K$ se plonge dans un corps de la forme~$k(\bh^{-1}\bS)$, o\`u~$\bh$ est une famille finie ind\'ependante de~$\rho(k)$.
\end{cor}
\begin{proof}
Choisissons une famille maximale~$\bT$ d'\'el\'ements homog\`enes de~$k$ alg\'ebriquement ind\'ependants sur~$k$. Notons~$\bg$ la famille des ordres. Le corps~$K$ est une extension finie de~$k(\bg^{-1}\bT)$. D'apr\`es le th\'eor\`eme d'Artin-Tate appliqu\'e aux $k$-alg\`ebres gradu\'ees $k \subset k(\bg^{-1}\bT)\subset K$, le corps gradu\'e~$k(\bg^{-1}\bT)$ est de type fini, d'o\`u l'assertion, compte tenu du lemme pr\'ec\'edent.
\end{proof}

\begin{rem}
A.~Ducros propose au th\'eor\`eme~2.7 de~\cite{variationdimension} une version du Nullstellensatz pour les alg\`ebres $k$-affino\"{\i}des, o\`u il donne une description explicite de celles qui sont des corps. La r\'eduction d'une telle alg\`ebre est une $\ti{k}$-alg\`ebre gradu\'ee de type fini et les corps gradu\'es qui apparaissent dans le Nullstellensatz gradu\'e sont exactement les r\'eductions des corps qu'il appelle de type~I. Si les r\'eductions des autres, ceux de type~II, sont absents de notre \'enonc\'e, c'est parce qu'ils sont des corps, mais non des corps valu\'es (sauf lorsqu'ils sont aussi de type~I) et que leurs r\'eductions ne sont pas des corps gradu\'es, ni m\^eme d'ailleurs des anneaux gradu\'es int\`egres.
\end{rem}

Nous allons maintenant d\'emontrer une propri\'et\'e g\'eom\'etrique qui nous sera utile dans la suite du texte : une vari\'et\'e gradu\'ee  int\`egre d\'efinie sur un corps alg\'ebriquement clos reste int\`egre apr\`es toute extension du corps de base. Commen\c{c}ons par un cas simple.

\begin{lem}
Soient~$k$ un corps gradu\'e et~$A$ une $k$-alg\`ebre gradu\'ee int\`egre. Soit~$\bg$ une famille d'\'el\'ements de~$G$ ind\'ependante de~$\rho(k)$. Alors l'alg\`ebre gradu\'ee $A\otimes_{k} k(\bg^{-1}\bT)$ est int\`egre.
\end{lem}
\begin{proof}
En proc\'edant par r\'ecurrence, il suffit de montrer le r\'esultat pour une famille~$\bg$ r\'eduite \`a un seul \'el\'ement~$g$. Dans ce cas, nous avons $k(g^{-1}T) \simeq k[g^{-1}T,gS]/(ST-1)$ et donc 
\[A\otimes_{k} k(g^{-1}T) \simeq A[g^{-1}T,gS]/(ST-1) \simeq A[g^{-1}T,gT^{-1}].\]
Un raisonnement sur les coefficients dominants permet de montrer que cette derni\`ere alg\`ebre gradu\'ee est int\`egre.
\end{proof}

En utilisant le Nullstellensatz gradu\'e, nous allons en d\'eduire le cas g\'en\'eral.

\begin{thm}\label{thm:geointegre}
Soient~$k$ un corps gradu\'e alg\'ebriquement clos et~$A$ une $k$-alg\`ebre gradu\'ee de type fini int\`egre. Alors, pour toute extension gradu\'ee~$K$ de~$k$, l'alg\`ebre gradu\'ee $A\otimes_{k} K$ est int\`egre.
\end{thm}
\begin{proof}
\'Ecrivons l'alg\`ebre~$A$ comme quotient d'un anneau de polyn\^omes $k[\bg^{-1} \bT] = k[g_{1}^{-1}T_{1},\ldots, g_{r}^{-1} T_{r}]$ par un id\'eal gradu\'e~$I$. Cet id\'eal est engendr\'e par des \'el\'ements homog\`enes $a_{1},\ldots,a_{n}$ de~$k[\bg^{-1} \bT]$. 

Soit~$K$ une extension gradu\'ee de~$k$. Nous pouvons supposer qu'elle est alg\'e\-bri\-quement close. Supposons, par l'absurde, que l'alg\`ebre gradu\'ee $A\otimes_{k} K = K[\bg^{-1} \bT]/(a_{1},\ldots,a_{n})$ n'est pas int\`egre. Il existe alors des polyn\^omes $\bg$-homog\`enes~$P$ et~$Q$ n'appartenant pas \`a l'id\'eal $(a_{1},\ldots,a_{n})$ et des polyn\^omes $\bg$-homog\`enes $b_{1},\ldots,b_{n}$ tels que l'on ait
\[PQ = \sum_{i=1}^n a_{i}b_{i} \textrm{ dans } K[\bg^{-1} \bT].\]

Il existe une extension gradu\'ee~$L$ de~$K$ et un \'el\'ement homog\`ene~$\bx$ de~$L^n$ tels que~$P(\bx)\ne 0$ et, pour tout~$i$, $a_{i}(\bx)=0$. Il existe \'egalement un \'el\'ement homog\`ene~$y$, que nous pouvons supposer appartenir au m\^eme~$L^n$, tel que \mbox{$Q(y)\ne 0$} et, pour tout~$i$, $a_{i}(y)=0$.

Soit~$B$ la sous-alg\`ebre gradu\'ee de~$K(\bg'^{-1}\bT')$ engendr\'ee sur~$k$ par les coordonn\'ees de~$\bx$ et de~$\by$, par $1/P(\bx)$, $1/P(\by)$, ainsi que les coefficients des polyn\^omes~$P$, $Q$, $b_{1},\ldots,b_{n}$. Par le Nullstellensatz gradu\'e, il existe une famille finie~$\bh$ d'\'el\'ements de~$G$, ind\'ependante de~$\rho(k)$, et un morphisme $\varphi : B \to k(\bh^{-1}\bS)$ qui induit l'identit\'e sur~$k$. Par construction, nous avons $\varphi(P)(\bx)\ne 0$, $\varphi(Q)(\by)\ne 0$, pour tout~$i$, $a_{i}(\bx)=a_{i}(\by)=0$, ainsi que l'\'egalit\'e
\[\varphi(P)\varphi(Q) = \sum_{i=1}^n a_{i}\varphi(b_{i}) \textrm{ dans } k(\bh^{-1}\bS)[\bg^{-1} \bT].\]
On en d\'eduit que l'anneau gradu\'e $A\otimes_{k}k(\bh^{-1}\bS)$ n'est pas int\`egre, ce qui contredit le lemme qui pr\'ec\`ede.
\end{proof}

\section{Points universels}\label{section:universel}

Dans cette section, nous nous int\'eressons aux morphismes d'extension des scalaires entre espaces analytiques, du type $X_{K} \to X_{k}$. Plus pr\'ecis\'ement, nous cherchons \`a comprendre sous quelles conditions les points de~$X_{k}$ peuvent se relever canoniquement \`a~$X_{K}$.

Commen\c{c}ons par signaler qu'en ce qui concerne les normes sur les produits tensoriels de modules norm\'es sur un anneau norm\'e, nous suivons les conventions habituelles du domaine (\emph{cf}~\cite{BGR}, \S 2.1.7 ou~\cite{rouge}, fin du \S 1.1) : si~$\As$ d\'esigne un anneau norm\'e et~$M$ et~$N$ deux $\As$-modules semi-norm\'es, nous d\'efinissons une semi-norme sur le produit tensoriel $M\otimes_{\As} N$ en posant, pour~$f$ dans~$M\otimes_{\As} N$,
\[\|f\| = \inf(\max_{i}(\|m_{i}\|\, \|n_{i}\|)),\]
la borne inf\'erieure portant sur l'ensemble des r\'epr\'esentations de~$f$ sous la forme $\sum_{i} m_{i}\otimes n_{i}$, avec $m_{i}\in M$ et $n_{i}\in N$. Nous d\'efinissons ensuite une $\As$-alg\`ebre de Banach~$M\ho_{\As} N$ en compl\'etant~$M\otimes_{\As} N$ par rapport \`a cette semi-norme.

D\'emontrons, \`a pr\'esent, un lemme technique, qui nous sera utile \`a plusieurs reprises par la suite.

\begin{lem}\label{lem:isometrie2}
Soit $A \to B$ un morphisme isom\'etrique de $k$-espaces vectoriels norm\'es. Soit~$C$ un $k$-espace vectoriel norm\'e. Alors le morphisme canonique \mbox{$A\otimes_{k} C \to B \otimes_{k} C$} est encore une isom\'etrie.

En particulier, le morphisme canonique $A\ho_{k} C \to B \ho_{k} C$ est une isom\'etrie.
\end{lem}
\begin{proof}
Si~$r$ est un nombre r\'eel n'appartenant pas \`a~$\sqrt{|k^*|}$, pour tout $k$-espace vectoriel norm\'e~$V$, l'injection canonique $V \to V\otimes_{k}k_{r}$ est isom\'etrique. Quitte \`a \'etendre les scalaires \`a~$k_{r}$, avec $r\notin \sqrt{|k^*|}$, nous pouvons donc supposer que la valuation du corps~$k$ n'est pas triviale. Quitte \`a remplacer~$A$ par son image dans~$B$, nous pouvons \'egalement supposer que $A \subset B$.

Nous noterons~$\|.\|_{A}$ et~$\|.\|_{B}$ les normes tensorielles sur~$A\otimes_{k} C$ et~$B \otimes_{k} C$. Soit~$f$ un \'el\'ement de~$A\otimes_{k} C$. Nous souhaitons montrer que $\|f\|_{A} = \|f\|_{B}$ et, pour ce faire, il suffit de montrer que $\|f\|_{A} \le \|f\|_{B}$, l'autre in\'egalit\'e \'etant imm\'ediate. 

Soit~$\eps > 0$. Il existe des \'el\'ements $b_{1},\ldots,b_{d}$ de~$B$ et $c_{1},\ldots,c_{d}$ de~$C$ tels que
\[f = \sum_{n=1}^d b_{n} \otimes c_{n} \textrm{ dans } B\otimes C\]
et 
\[\|f\|_{B} \le \max_{1\le n\le d} (\|b_{n}\|\, \|c_{n}\|) \le \|f\|_{B} + \eps.\]

Soit~$A_{0}$ un sous-espace vectoriel de dimension finie de~$A$ tel que $f\in A_{0}\otimes_{k} C$. Soit~$B_{0}$ un sous-espace vectoriel de dimension finie de~$B$ contenant~$A_{0}$ et les~$b_{i}$. Soit $\alpha>1$. D'apr\`es~\cite{BGR}, proposition 2.6.2/3, il existe deux entiers $s\ge r$ et une famille~$(\beta_{i})_{1\le i\le r}$ d'\'el\'ements de~$B_{0}$ tels que $(\beta_{1},\ldots,\beta_{s})$ soit une base de~$A_{0}$ et $(\beta_{1},\ldots,\beta_{r})$ soit une base $\alpha$-cart\'esienne de~$B_{0}$. Rappelons que cette derni\`ere condition signifie que pour tout \'el\'ement $b = \sum_{i=1}^r \lambda_{i} \beta_{i}$ de~$B_{0}$, nous avons 
\[\max_{1\le i\le r} (|\lambda_{i}|\, \|\beta_{i}\|) \le \alpha \|b\|.\]

\'Ecrivons chacun des~$b_{n}$ sous la forme 
$b_{n} = \sum_{i=1}^r \lambda_{n,i} \beta_{i}$. Nous avons alors
\begin{align*}
 \max_{1\le n\le r} (\|b_{n}\|\, \|c_{n}\|) &\ge \alpha^{-1}\, \max_{1\le n\le d, 1\le i\le r} (|\lambda_{n,i}|\, \|\beta_{i}\|\, \|c_{n}\|)\\
&\ge \alpha^{-1}\, \max_{1\le n\le d, 1\le i \le s} (|\lambda_{n,i}|\, |\beta_{i}|\, \|c_{n}\|)\\
&\ge \alpha^{-1}\, \|f\|_{A}
\end{align*}
car $f = \sum_{n=1}^d \left(\sum_{i=1}^s \lambda_{n,i} \beta_{i}\right) \otimes c_{n}$ dans~$A\otimes_{k} C$. Par cons\'equent, nous avons
\[\|f\|_{A} \le \alpha(\|f\|_{B} + \eps).\]
En utilisant le fait que cette in\'egalit\'e vaut pour tout $\alpha>1$, puis pour tout~$\eps > 0$, on en d\'eduit le r\'esultat attendu.
\end{proof}

Pr\'ecisons maintenant la notion de rel\`evement canonique \'evoqu\'ee plus haut.

\begin{defi}
Une $k$-alg\`ebre de Banach commutative et unitaire~$\As$ est dite {\bf universellement multiplicative} sur~$k$ si, pour toute extension valu\'ee compl\`ete~$K$ de~$k$, la norme de l'alg\`ebre $\As \ho_{k} K$ est multiplicative.

Soient~$\As$ une $k$-alg\`ebre de Banach commutative et unitaire. On dit qu'un point~$x$ de~$X=\Mc(\As)$ est {\bf universel} sur~$k$ si son corps r\'esiduel compl\'et\'e~$\Hs(x)$ est universellement multiplicatif sur~$k$. Nous noterons~$X_{u}$ l'ensemble des points universels de~$X$ sur~$k$.
\end{defi}

Soient~$\As$ une $k$-alg\`ebre de Banach commutative et unitaire. Soient~$x$ un point de~$X=\Mc(\As)$ et~$K$ une extension valu\'ee compl\`ete de~$k$. Si le point~$x$ est universel ou si le corps~$K$ est universel, nous disposons d'un point canonique de~$X_{K}$ au-dessus de~$x$ : celui qui correspond \`a la norme de l'alg\`ebre $\Hs(x)\hat{\otimes}_{k} K$. Nous le noterons~$\sigma_{K/k}(x)$, ou simplement~$\sigma_{K}(x)$ si aucune confusion n'en r\'esulte. 
 
Remarquons que pour $x\in X$ et $k \subset L \subset K$, lorsque ces quantit\'es sont d\'efinies, nous avons $\sigma_{K/k}(x) = \sigma_{K/L}(\sigma_{L/k}(x))$.

\begin{rem}
Notre notion de point universel n'est autre que celle de \og peaked point \fg\ d\'efinie par V.~Berkovich dans~\cite{rouge}, \S 5.2 (qui note donc~$X_{p}$ l'ensemble que nous notons~$X_{u}$). Si la notion de pic rend bien compte du comportement du point dans sa fibre apr\`es extension des scalaires, elle nous semble trompeuse lorsque l'on s'int\'eresse \`a l'ensemble de ces points. Nous montrerons par exemple \`a la fin de cette section que, sur un corps alg\'ebriquement clos, tout point est un pic, ce qui est difficilement conciliable avec l'intuition. 
\end{rem}

Nous g\'en\'eralisons ici le lemme~5.2.6 et le corollaire~5.2.7 de~\cite{rouge}. Rappelons tout d'abord une d\'efinition.

\begin{defi}
Un $k$-espace de Banach est dit {\bf de type d\'enombrable} sur~$k$ s'il contient un sous-$k$-espace vectoriel dense de dimension d\'enombrable.
\end{defi}

Nous utiliserons surtout cette d\'efinition pour des extensions valu\'ees compl\`etes du corps~$k$. C'est donc dans ce sens qu'il faudra entendre extension de type d\'enombrable du corps~$k$.

\begin{rem}
On pourrait d\'efinir de fa\c{c}on analogue la notion de $k$-espace de Banach de type fini, mais celle-ci co\"{\i}ncide avec la notion de $k$-espace de Banach de dimension finie, d'apr\`es~\cite{BGR}, proposition~2.3.3/4.
\end{rem}

\begin{lem}\label{lem:continuitenorme}
Soient~$\As$ une $k$-alg\`ebre de Banach commutative et unitaire et~$B$ un $k$-espace de Banach. Soit~$\varphi$ un \'el\'ement de~$\As\ho_{k} B$. Pour tout point~$x$ de~$\Mc(\As)$, notons~$\|.\|_{x,B}$ la norme tensorielle sur~$\Hs(x)\ho_{k} B$. Alors la fonction $x\in \Mc(\As) \mapsto \|\varphi\|_{x,B}$ est continue.
\end{lem}
\begin{proof}
Par d\'efinition, $\varphi$ est limite d'une suite d'\'el\'ements de~$\As\otimes_{k} B$. Dans chacun de ces \'el\'ements n'intervient qu'un nombre fini d'\'el\'ements de~$B$. Il existe donc un sous-$k$-espace de Banach~$B_{0}$ de~$B$, de type d\'enombrable sur~$k$, tel que $\varphi\in \As\ho_{k} B_{0}$. D'apr\`es le lemme~5.2.6 de~\cite{rouge}, la fonction $x\in \Mc(\As) \mapsto \|\varphi\|_{x,B_{0}}$ est continue. Mais, par le lemme~\ref{lem:isometrie2}, pour tout point~$x$ de~$\Mc(\As)$, nous avons $\|\varphi\|_{x,B_{0}}=\|\varphi\|_{x,B}$.
\end{proof}

\begin{cor}\label{cor:remonte}
Soient~$\As$ une $k$-alg\`ebre de Banach commutative et unitaire et~$K$ une extension valu\'ee compl\`ete (resp. universelle) du corps~$k$. Notons~$X=\Mc(\As)$. L'application {$\sigma_{K} : X_{u} \to X\hat{\otimes}_{k} K$} (resp. $\sigma_{K} : X \to X\hat{\otimes}_{k} K$) est continue.
\end{cor}

\begin{rem}
Une autre possibilit\'e pour d\'efinir l'universalit\'e, peut-\^etre plus naturelle, consisterait \`a demander qu'un point soit universel non pas lorsque sa norme est universellement multiplicative, mais lorsque son rayon spectral l'est. Nous ne sommes malheureusement pas parvenu \`a obtenir un analogue du corollaire~\ref{cor:remonte} dans ce cadre.

Il faut donc prendre garde au fait qu'il ne suffit pas qu'un point se rel\`eve de fa\c{c}on canonique dans toute extension des scalaires pour qu'il soit universel. Consid\'erons, par exemple, un point~$x$ dont le corps r\'esiduel~$\Hs(x)$ est une extension purement ins\'eparable du corps~$k$. Bien qu'il se rel\`eve canoniquement dans toute extension des scalaires, il n'est pas universel : l'anneau $\Hs(x)\ho_{k} \Hs(x)$ n'\'etant pas m\^eme r\'eduit, il ne saurait \^etre muni d'une norme multiplicative.
\end{rem}

Nous allons maintenant chercher des crit\`eres permettant d'assurer qu'un point est universel. Dans le r\'esultat qui suit, nous \'etendons, dans certains cas, le r\'esultat de la proposition~5.2.5 de~\cite{rouge}. Signalons que les r\'eductions dont il est question ici sont les r\'eductions gradu\'ees au sens de M.~Temkin (\emph{cf.} section~\ref{section:alggrad}).

\begin{prop}\label{prop:Shilovuniversel}
Soit~$\As$ une alg\`ebre $k$-affino\"{\i}de. Supposons que le corps~$k$ est stable et que $\rho(\As) \cap \sqrt{|k^*|} \subset |k^*|$. Supposons que la r\'eduction gradu\'ee~$\ti{X}$ de l'espace $k$-affino\"{\i}de $X=\Mc(\As)$ est g\'eom\'etriquement int\`egre (au sens o\`u pour toute extension gradu\'ee~$\ti{K}$ de~$\ti{k}$, l'anneau gradu\'e $\ti{\As}\otimes_{\ti{k}}\ti{K}$ est int\`egre). Alors le bord de Shilov de~$X$ est un singleton et son unique point est universel.
\end{prop}
\begin{proof}
Nous pouvons supposer que l'alg\`ebre~$\As$ est r\'eduite. Puisque l'alg\`ebre gradu\'ee~$\ti{\As}$ est int\`egre, d'apr\`es~\cite{TemkinII}, proposition~3.3, le bord de Shilov de~$X$ est r\'eduit \`a un point, que nous noterons~$\gamma$. Soit~$K$ une extension valu\'ee compl\`ete du corps~$k$. Nous souhaitons montrer que la norme tensorielle sur l'alg\`ebre~$\Hs(\gamma)\ho_{k} K$ est multiplicative. 

Les hypoth\`eses assurent qu'il existe un polyrayon $k$-libre~$\br$ tel que l'alg\`ebre~$\As\ho_{k} k_{\br}$ soit strictement $k_{\br}$-affino\"{\i}de et que $\rho(\As\ho_{k} k_{\br}) = |k_{\br}|$. 

Puisque le corps~$k$ est stable, le corps~$k_{\br}$ l'est aussi (\emph{cf.}~\cite{stablemodification}, corollary~6.3.6 ou~\cite{RSS}, th\'eor\`eme~3.3.16). L'alg\`ebre~$\As\ho_{k} k_{\br}$ est r\'eduite et le th\'eor\`eme~6.4.3/1 de~\cite{BGR} assure qu'elle est distingu\'ee. Notons~$\gamma_{\br}$ l'unique point du bord de Shilov de~$X\ho_{k} k_{\br}$. En utilisant le fait que l'\'el\'ement~$r$ de~$\R_{+}^*$ n'appartient pas \`a $\sqrt{|\Hs(\gamma)^*|}=\sqrt{\rho(\As)^*}$, on montre que le morphisme naturel $\Hs(\gamma)\ho_{k} k_{\br} \to \Hs(\gamma_{\br})$ est un isomorphisme isom\'etrique.

D'apr\`es la proposition~5.2.5 de~\cite{rouge} (noter la correction apport\'ee \`a sa preuve, ou plut\^ot \`a celle du lemme~5.2.2 sur lequel elle repose, par le lemme~1.8 de~\cite{excellence}), le point~$\gamma_{\br}$ est universel et la norme tensorielle sur l'alg\`ebre $\Hs(\gamma_{\br}) \ho_{k_{\br}} K_{\br}$ est donc multiplicative. Or $\Hs(\gamma_{\br}) \ho_{k_{\br}} K_{\br} = (\Hs(\gamma)\ho_{k} K)\ho_{K} K_{\br}$ et le lemme pr\'ec\'edent utilis\'e avec $K \subset K_{\br}$ permet de conclure.
\end{proof}

\begin{rem}
Le recours \`a un changement de base dans la preuve pr\'e\-c\'e\-dente peut sembler artificiel. Une m\'ethode plus naturelle consisterait \`a utiliser une notion plus g\'en\'erale d'alg\`ebre distingu\'ee, autorisant des alg\`ebres de Tate~$k\{\br^{-1}\bT\}$ avec des rayons arbitraires, et \`a prouver sur celle-ci les r\'esultats classiques de~\cite{Orthonormalbasen}. Signalons que des probl\`emes li\'es \`a cette question apparaissent d\'ej\`a dans le cas strictement affino\"{\i}de puisque, pour $r \in \sqrt{|k^*|} \setminus |k^*|$, l'alg\`ebre $k\{r^{-1}T\}$ n'est pas distingu\'ee, mais le devient apr\`es extension des scalaires au corps~$k_{r}$, le compl\'et\'e du corps des fractions de $k\{r^{-1}T\}$.
\end{rem}

\begin{cor}\label{cor:Shilovuniverselalgclos}
Soit~$X$ un espace $k$-affino\"{\i}de. Supposons que le corps~$k$ est alg\'ebriquement clos. Alors tout point du bord de Shilov de~$X$ est universel.
\end{cor}
\begin{proof}
Soit~$\gamma$ un point du bord de Shilov de~$X=\Mc(\As)$. D'apr\`es~\cite{TemkinII}, proposition~3.3, sa r\'eduction gradu\'ee~$\ti{\gamma}$ est un point g\'en\'erique de~$\ti{X}$. Choisissons un \'el\'ement~$\ti{f}$ de~$\ti{\As}$ qui s'annule en tous les points g\'en\'eriques de la r\'eduction gradu\'ee~$\ti{X}$ \`a l'exception de~$\ti{\gamma}$ et relevons-le en un \'el\'ement~$f$ de~$\As$. Posons~$r=\rho(f)$. D'apr\`es~\cite{TemkinII}, proposition~3.1, la r\'eduction gradu\'ee de l'alg\`ebre affino\"{\i}de~$\As\{r^{-1}f\}$ est \'egale \`a~$\ti{A}_{\ti{f}}$ et le bord de Shilov de son spectre~$Y$ est donc r\'eduit au point~$\gamma$. 

Nous pouvons maintenant conclure en appliquant le r\'esultat de la proposition~\ref{prop:Shilovuniversel} \`a l'espace~$Y$. Les hypoth\`eses sur le corps~$k$ sont satisfaites : \'etant alg\'ebriquement clos, il est stable et son groupe des valeurs est divisible. Celle sur~$Y$ l'est \'egalement puisque sa r\'eduction gradu\'ee~$\ti{Y}=\ti{X}_{\ti{f}}$ est int\`egre et donc g\'eom\'etriquement int\`egre, d'apr\`es le th\'eor\`eme~\ref{thm:geointegre}.
\end{proof}

Il nous sera \'egalement possible d'obtenir des points universels comme limites de familles de points universels.

\begin{prop}\label{prop:impliqueuniversel}
Soit~$\As$ une $k$-alg\`ebre de Banach commutative et unitaire. L'ensemble des points universels de~$\Mc(\As)$ est ferm\'e.
\end{prop}
\begin{proof}
Soit~$\gamma$ un point de~$X=\Mc(\As)$. Supposons qu'il existe un ensemble ordonn\'e filtrant~$(I,<)$ et une suite g\'en\'eralis\'ee~$(\gamma_{i})_{i\in I}$ de points universels de~$X$ qui converge vers~$\gamma$. 

Soit~$K$ une extension valu\'ee compl\`ete du corps~$k$. Nous souhaitons montrer que la norme tensorielle sur $\Hs(\gamma)\ho_{k} K$ est multiplicative. D'apr\`es le lemme~\ref{lem:isometrie2}, il suffit  de montrer que la norme tensorielle sur $k(\p_{\gamma})\otimes_{k} K$, est multiplicative. Nous noterons~$\|.\|_{\gamma,K}$ cette norme. Nous d\'efinissons de m\^eme une norme~$\|.\|_{\gamma_{i},K}$, pour tout~$i$.

D'apr\`es le lemme~\ref{lem:continuitenorme}, pour tout \'el\'ement~$\varphi$ de~$\As\ho_{k} K$, la suite g\'en\'eralis\'ee~$\|\varphi\|_{\gamma_{i},K}$ tend vers~$\|\varphi\|_{\gamma,K}$. Puisque les normes~$\|.\|_{\gamma_{i},K}$ sont multiplicatives, la norme~$\|.\|_{\gamma,K}$ est multiplicative sur~$\As\ho_{k} K$.

Soient~$f_{1}$ et~$f_{2}$ deux \'el\'ements de~$k(\p_{\gamma})\otimes_{k} K$. Il existe des \'el\'ements~$g_{1}$ et~$g_{2}$ de~$\As\otimes_{k} K$ et des \'el\'ements~$u_{1}$ et~$u_{2}$ de~$\As$ ne s'annulant pas en~$\gamma$ tels que 
\[f_{1} = u_{1}^{-1}(\gamma) g_{1} \textrm{ et } f_{2} = u_{2}^{-1}(\gamma) g_{2} \textrm{ dans } k(\p_{\gamma})\otimes_{k} K.\]
Nous avons donc
\[\renewcommand{\arraystretch}{1.3}{\begin{array}{rcl}
\|f_{1}f_{2}\|_{\gamma,K} &=& \|u_{1}^{-1}(\gamma)u_{2}^{-1}(\gamma)g_{1}g_{2}\|_{\gamma,K}\\
&=& |u_{1}^{-1}(\gamma)u_{2}^{-1}(\gamma)|\, \|g_{1}g_{2}\|_{\gamma,K}\\
&=& |u_{1}^{-1}(\gamma)|\, |u_{2}^{-1}(\gamma)|\, \|g_{1}\|_{\gamma,K}\, \|g_{2}\|_{\gamma,K}\\
&=& \|f_{1}\|_{\gamma,K}\, \|f_{2}\|_{\gamma,K}.
\end{array}}\]
Le r\'esultat s'ensuit.
\end{proof}

\begin{rem}
Les notions et r\'esultats de cette section s'\'etendent sans peine du cas des spectres de $k$-alg\`ebres de Banach commutatives et unitaires \`a celui des espaces $k$-analytiques. Ainsi, un point~$x$ de~$X$ sera-t-il dit universel sur~$k$ si son corps r\'esiduel compl\'et\'e~$\Hs(x)$ l'est. Quant au corollaire~\ref{cor:remonte} et \`a la proposition~\ref{prop:impliqueuniversel}, ils se g\'en\'eralisent ais\'ement, car les r\'esultats pour l'espace~$X$ d\'ecoulent des m\^emes r\'esultats pour ses domaines affino\"{\i}des.
\end{rem}

\begin{cor}\label{cor:cime}
Tout point d'un espace analytique sur un corps alg\'ebriquement clos est universel.
\end{cor}
\begin{proof}
Il suffit de d\'emontrer que l'ensemble des points universels est dense. Cela d\'ecoule du corollaire~\ref{cor:Shilovuniverselalgclos} appliqu\'e aux domaines affino\"{\i}des de l'espace consid\'er\'e.
\end{proof}

\begin{rem}
Si l'on ne s'int\'eresse qu'aux espaces analytiques sur un corps de valuation non triviale, il est possible de d\'emontrer ce r\'esultat en utilisant uniquement les r\'eductions classiques. En effet, on se ram\`ene d'abord au cas des espaces affino\"{\i}des, puis \`a celui des disques (car les espaces affino\"{\i}des en sont des ferm\'es de Zariski) et \`a celui des espaces affines. Dans ces derniers, tout point poss\`ede un syst\`eme fondamental de voisinages form\'e de domaines strictement $k$-affino\"{\i}des, ce qui permet de conclure.
\end{rem}

\begin{cor}
Soient~$X$ un espace $k$-analytique et~$x$ un point de~$X$. Notons~$F$ le compl\'et\'e d'une cl\^oture alg\'ebrique de~$k$. Les propri\'et\'es suivantes sont \'equivalentes :
\begin{enumerate}[\it i)]
\item le point~$x$ est universel\ ;
\item la norme tensorielle sur $\Hs(x)\ho_{k} F$ est multiplicative.
\end{enumerate}
\end{cor}
\begin{proof}
L'implication {\it i)} $\implies$ {\it ii)} est imm\'ediate.

Montrons que {\it ii)} $\implies$ {\it i)} et, pour cela, supposons que la norme tensorielle sur $\Hs(x)\ho_{k} F$ est multiplicative. Dans ce cas, la fibre du morphisme $X_{F} \to X_{k}$ au-dessus de~$x$ poss\`ede un unique point~$x_{F}$ dans son bord de Shilov et le morphisme d'\'evaluation $\Hs(x)\ho_{k} F \to \Hs(x_{F})$ est une isom\'etrie. 

Soit~$K$ une extension valu\'ee compl\`ete de~$k$. Nous voulons montrer que la norme tensorielle sur $\Hs(x)\ho_{k} K$ est multiplicative. D'apr\`es le lemme~\ref{lem:isometrie2}, quitte \`a remplacer le corps~$K$ par le compl\'et\'e de sa cl\^oture alg\'ebrique, nous pouvons supposer qu'il est alg\'ebriquement clos. Nous pouvons donc supposer qu'il contient le corps~$F$. D'apr\`es le corollaire qui pr\'ec\`ede, le point~$x_{F}$ est universel et la norme tensorielle sur $\Hs(x_{F})\ho_{F} K$ est multiplicative. En utilisant de nouveau le lemme~\ref{lem:isometrie2}, on en d\'eduit que la norme tensorielle sur $(\Hs(x)\ho_{k} F)\ho_{F} K = \Hs(x)\ho_{k} K$ est multiplicative.
\end{proof}

\section{Points des disques}\label{section:universeldisque}

Dans cette section, nous nous consacrons plus sp\'ecifiquement aux points des disques. Nous montrons qu'ils peuvent \^etre d\'efinis sur des corps de type d\'enombrable, ce qui nous permettra d'effectuer l'op\'eration de \og descente \fg\ d\'ecrite dans l'introduction.

Nous commencerons par \'etudier une famille particuli\`ere de points. 

\begin{defi}
Soit~$X$ un espace $k$-analytique. Un point~$x$ de~$X$ est appel\'e {\bf point d'Abhyankar} s'il satisfait l'\'egalit\'e $s_{k}(x)+t_{k}(x)=\dim_{k,x}(X)$. 
\end{defi}

\begin{rem}
La description explicite des points de la droite affine~$\E{1}{k}$ donn\'ee par V.~Berkovich montre que les points d'Abhyankar de~$\D^1_{k}(r)$, avec \mbox{$r>0$}, sont exactement ceux de type~2 ou~3.
\end{rem}

\begin{rem}\label{rem:p=0}
L'in\'egalit\'e d'Abhyankar rappel\'ee \`a la fin de la section~\ref{section:intro} montre qu'un point d'Abhyankar~$x$ d'un espace analytique irr\'eductible ne peut appartenir \`a aucun ferm\'e analytique. %En particulier, en un tel point, nous avons \mbox{$\p_{x}=(0)$}.
\end{rem}

Le r\'esultat qui suit permet d'exhiber des points d'Abhyankar.

\begin{lem}\label{lem:Shilovstaff}
Soient~$X$ un espace strictement $k$-affino\"{\i}de et~$x$ un point de son bord de Shilov. Alors $s_{k}(x) = \dim_{k,x}(X)$.
\end{lem}
\begin{proof}
Posons~$d = \dim_{k,x}(X)$. Nous pouvons supposer que~$X$ est irr\'eductible de dimension~$d$. Notons~$\As$ l'alg\`ebre de l'espace~$X$. Soit~$f$ un \'el\'ement de~$\As^\circ$ dont la r\'eduction~$\ti{f}$ est nulle sur toutes les composantes irr\'eductibles de~$\ti{X}$ ne contenant pas~$\ti{x}$, mais pas en~$\ti{x}$. Consid\'erons le domaine affino\"{\i}de~$V$ de~$X$ d\'efini par~$\{|f|=1\}$. D'apr\`es~\cite{BGR}, proposition~7.2.6/3 (et~\cite{TemkinII}, proposition~3.1 dans le cas de valuation triviale), sa r\'eduction est isomorphe \`a~$D(\ti{f})$. Remarquons que~$V$ et sa r\'eduction sont de dimension~$d$.

Le point~$x$ est l'unique image r\'eciproque par l'application de r\'eduction $V \to \ti{V}$ du point g\'en\'erique~$\ti{x}$ de $\ti{V}=D(\ti{f})$. Par cons\'equent, nous avons une injection $\ti{k}(\ti{x}) \hookrightarrow \wti{\Hs(x)}$. On en d\'eduit que~$s_{k}(x)$ est sup\'erieur \`a~$d$, et donc \'egal \`a~$d$, par l'in\'egalit\'e d'Abhyankar.
\end{proof}

\begin{cor}\label{cor:smaxdense}
Supposons que la valuation de~$k$ n'est pas triviale. Soit~$X$ un espace strictement $k$-affino\"{\i}de \'equidimensionnel de dimension~$d$. Alors, l'ensemble des points~$x$ tels que $s_{k}(x)=d$ est dense dans~$X$.
\end{cor}

Pour \'etendre ces r\'esultat au cas g\'en\'eral, nous aurons besoin d'un lemme.

\begin{lem}\label{lem:extkr}
Soit~$X$ un espace $k$-affino\"{\i}de. Soit~$r$ un nombre r\'eel strictement positif n'appartenant pas \`a~$\sqrt{|k^*|}$. Consid\'erons l'espace $k_{r}$-affino\"{\i}de $X_{r} = X\ho_{k} k_{r}$ et notons~$\pi$ sa projection sur~$X$. Le bord de Shilov de la fibre~$\pi^{-1}(x)$ contient un unique point, que nous noterons~$x_{r}$. Nous avons
\begin{enumerate}[\it i)]
\item $s_{k_{r}}(x_{r})=s_{k}(x)+1$ et $t_{k_{r}}(x_{r})=t_{k}(x)-1$ si $r\in \sqrt{|\Hs(x)^*|}$ ;
\item $s_{k_{r}}(x_{r})=s_{k}(x)$ et $t_{k_{r}}(x_{r})=t_{k}(x)$ si $r\notin \sqrt{|\Hs(x)^*|}$.
\end{enumerate} 
\end{lem}
\begin{proof}
Notons~$\As$ l'alg\`ebre de~$X$. On v\'erifie que le point~$x_{r}$ est associ\'e \`a la semi-norme
\[\sum_{m\in\Z} a_{m} T^m \in \As\ho_{k} k_{r} \mapsto \max_{m\in\Z}(|a_{m}(x)| r^m).\]
Les r\'esultats \'enonc\'es s'en d\'eduisent.
\end{proof}

\begin{prop}
Tout point du bord de Shilov d'un espace affino\"{\i}de est un point d'Abhyankar.
\end{prop}
\begin{proof}
Soit~$X$ un espace $k$-affino\"{\i}de et~$x$ un point de son bord de Shilov. Puisque~$x$ appartient \`a une seule composante irr\'eductible de~$X$, quitte \`a remplacer~$X$ par cette composante, nous pouvons supposer que~$X$ est irr\'eductible. Soit~$\bs$ un polyrayon $k$-libre tel que $X\ho_{k} k_{\bs}$ soit strictement $k_{\bs}$-affino\"{\i}de. Par une r\'ecurrence utilisant le lemme qui pr\'ec\`ede, on associe au point~$x$ de~$X$ un point~$x_{\bs}$ appartenant au bord de Shilov de $X\ho_{k} k_{\bs}$ et v\'erifiant $s_{k_{\bs}}(x_{\bs})+t_{k_{\bs}}(x_{\bs})=s_{k}(x)+t_{k}(x)$. Le r\'esultat d\'ecoule alors du lemme~\ref{lem:Shilovstaff} et de l'invariance de la dimension par extension des scalaires.
\end{proof}

\begin{cor}\label{cor:Abhyankardense}
L'ensemble des points d'Abhyankar d'un espace analytique est dense.
\end{cor}

Dans la suite de cette section, nous montrerons que la r\'eciproque de la proposition est vraie, dans certains cas. Nous prouverons en fait un r\'esultat plus pr\'ecis en imposant des restrictions sur les domaines affino\"{\i}des qui interviennent. Commen\c{c}ons par des lemmes techniques. Le premier se d\'emontre en utilisant les m\^emes m\'ethodes que dans la preuve du lemme~\ref{lem:Shilovstaff}.

\begin{lem}\label{lem:equidim}
Soit~$X$ un espace strictement $k$-affino\"{\i}de \'equidimensionnel. Alors sa r\'eduction~$\ti{X}$ est \'equidimensionnelle et de m\^eme dimension.
\end{lem}

\begin{lem}\label{lem:Shilov}
Soit~$Y$ un espace $k$-affino\"{\i}de \'equidimensionnel d'alg\`ebre~$\Bs$. Soit~$P(T)$ un polyn\^ome unitaire non constant \`a coefficients dans~$\Bs$ et soit~$r>0$. Consid\'erons le domaine affino\"{\i}de~$X$ de~$\E{1}{Y}$ d\'efini par $\{|P(T)|\le r\}$. Notons~$\Gamma_{Y}$ le bord de Shilov de~$Y$. Pour tout point~$\gamma$ de~$\Gamma_{Y}$, notons~$\Gamma_{\gamma}$ le bord de Shilov de la fibre~$X_{\gamma}$. Alors le bord de Shilov de~$X$ est
\[\Gamma_{X} = \bigcup_{\gamma\in\Gamma_{Y}} \Gamma_{\gamma}.\] 
\end{lem}
\begin{proof}
Quitte \`a \'etendre les scalaires \`a un corps~$k_{\br}$, o\`u~$\br$ d\'esigne un polyrayon $k$-libre bien choisi, nous pouvons supposer que les espaces~$X$ et~$Y$ sont strictement $k$-affino\"{\i}des et que la valuation du corps~$k$ n'est pas triviale. Cette op\'eration pr\'eserve le caract\`ere \'equidimensionnel de~$Y$, comme on le voit en se ramenant au cas irr\'eductible et r\'eduit et en constatant que l'alg\`ebre~$\Bs\ho_{k} k_{\br}$ est alors int\`egre.

Notons~$Z$ le domaine strictement affino\"{\i}de de~$\E{1}{Y}$ d\'efini par $\{|T|\le r\}$. Une description explicite de l'alg\`ebre de~$Z$ montre que son bord de Shilov~$\Gamma_{Z}$ n'est autre que $\Gamma_{Z} =\{z_{\gamma},\ \gamma\in\Gamma_{Y}\}$, o\`u, pour tout~$\gamma\in \Gamma_{Y}$, $z_{\gamma}$~d\'esigne l'unique point du bord de Shilov de la fibre~$Z_{\gamma}$. Consid\'erons le morphisme fini $\varphi : X \to Z$ d\'efini par le polyn\^ome~$P$. D'apr\`es~\cite{BGR}, 6.3.5/1, le morphisme induit $\ti{\varphi} : \ti{X} \to \ti{Z}$ est fini. 

Puisque~$X$ est strictement $k$-affino\"{\i}de, d'apr\`es~\cite{rouge}, proposition~2.4.4, un point~$x$ appartient \`a son bord de Shilov~$\Gamma_{X}$ si, et seulement si, sa r\'eduction~$\ti{x}$ est un point g\'en\'erique de~$\ti{X}$, donc, d'apr\`es le lemme qui pr\'ec\`ede, si, et seulement si, le degr\'e de transcendance sur~$\ti{k}$ du corps~$\ti{k}(\ti{x})$ est \'egal \`a~$\dim(X)$. Le m\^eme r\'esultat vaut pour~$Z$. En utilisant le fait qu'un morphisme fini pr\'eserve le degr\'e de transcendance, on en d\'eduit que $\Gamma_{X} = \varphi^{-1}(\Gamma_{Z})$, puis le r\'esultat attendu.
\end{proof}

Par les m\^emes m\'ethodes, on peut d\'emontrer un r\'esultat d\'ecrivant le bord de Shilov d'un ferm\'e de Zariski d'une droite relative. 

\begin{lem}
Soit~$Y$ un espace $k$-affino\"{\i}de \'equidimensionnel d'alg\`ebre~$\Bs$. Soit~$P(T)$ un polyn\^ome unitaire non constant \`a coefficients dans~$\Bs$. Consid\'erons le ferm\'e de Zariski~$X$ de~$\E{1}{Y}$ d\'efini par $\{P(T)=0\}$. Notons~$\Gamma_{Y}$ le bord de Shilov de~$Y$. Alors le bord de Shilov de~$X$ est
\[\Gamma_{X} = \bigcup_{\gamma\in\Gamma_{Y}} X_{\gamma}.\]
\end{lem}

\begin{lem}
Soient~$r>0$ et~$x$ un point de type~$2$ ou~$3$ de~$\D^1(r)$. Il existe un polyn\^ome~$P(T)$ \`a coefficients dans~$k$ tel que le domaine affino\"{\i}de d\'efini par $\{|P(T)|\le |P(T)(x)|\}$ ait un bord de Shilov r\'eduit au point~$x$. 
\end{lem}
\begin{proof}
Si le point~$x$ est l'unique point~$\eta_{r}$ du bord de Shilov du disque, le r\'esultat est imm\'ediat. Sinon, consid\'erons une composante connexe de $\D^1(r)\setminus\{x\}$ ne contenant pas~$\eta_{r}$. C'est un ouvert de~$\D^1(r)$ qui contient un point rigide~$y$ d\'efini par l'annulation d'un polyn\^ome irr\'eductible~$P(T)$ \`a coefficients dans~$k$. Remarquons que le point~$x$ se trouve sur l'unique chemin joignant les points~$y$ et~$\eta_{r}$. 

Soit~$Q(T)$ un polyn\^ome irr\'eductible \`a coefficients dans~$k$. Rappelons que nous savons comment se comporte la valeur absolue~$|Q(T)(z)|$ lorsque le point~$z$ varie sur la droite~$\P^{1,\mathrm{an}}_{k}$. Notons~$q$ l'unique point de la droite en lequel~$Q(T)$ s'annule. Alors la valeur absolue~$|Q(T)(z)|$ cro\^{\i}t strictement lorsque le point~$z$ parcourt le segment~$[q,\infty]$ et est localement constante sur son compl\'ementaire. En utilisant ce fait ainsi que la densit\'e des polyn\^omes dans l'alg\`ebre du disque~$\D^1(r)$, on v\'erifie que~$x$ est l'unique point du bord de Shilov du domaine affino\"{\i}de \mbox{$\{|P(T)|\le |P(T)(x)|\}$}.
\end{proof}

\begin{prop}\label{prop:s+tmax}
Soient~$p\in\N$ et~$\br \in (\R_{+}^*)^p$. Soit~$x$ un point d'Abhyankar de~$\D^p_{k}(\br)$. Alors il existe un domaine de Weierstra\ss~$V$ de~$\D^p_{k}(\br)$ dont le bord de Shilov est r\'eduit au point~$x$. En outre, le domaine affino\"{\i}de~$V$ peut \^etre d\'efini sur un sous-corps de~$k$ de type fini sur~$k_{p}$.
\end{prop}
\begin{proof}
D\'emontrons ce r\'esultat par r\'ecurrence sur l'entier~$p$. Si $p=0$, c'est \'evident.

Supposons le r\'esultat vrai pour~$p$ et d\'emontrons-le pour~$p+1$. Soit $\br = (r_{1},\ldots,r_{p+1}) \in (\R_{+}^*)^{p+1}$. Posons $\br'=(r_{1},\ldots,r_{p})$. Soit~$x$ un point d'Abhyankar de~$\D^{p+1}_{k}(\br)$. Consid\'erons le morphisme de projection sur les~$p$ premi\`eres coordonn\'ees $\pi : \D^{p+1}_{k}(\br) \to \D^p_{k}(\br')$. Posons $y=\pi(x)$. 

Nous avons $s(x) = s(\Hs(x)/\Hs(y)) + s(y)$ et $t(x) = t(\Hs(x)/\Hs(y)) + t(y)$. L'in\'egalit\'e d'Abhyankar assure que $s(y)+t(y)\le p$  et que $s(\Hs(x)/\Hs(y)) + t(\Hs(x)/\Hs(y)) \le 1$ (en l'appliquant au point~$y$ de~$\D^1_{\Hs(y)}(r_{p+1})$). On en d\'eduit que~$y$ est un point d'Abhyankar de~$\D^{p}_{k}(\br')$ et que~$x$ est un point d'Abhyankar (c'est-\`a-dire un point de type~$2$ ou~$3$) de $\D^1_{\Hs(y)}(r_{p+1})$.

D'apr\`es l'hypoth\`ese de r\'ecurrence, il existe un domaine de Weierstra\ss~$W$ de~$\D^p_{k}(\br)$ dont le bord de Shilov est r\'eduit au point~$y$. Notons~$\As_{W}$ l'alg\`ebre affino\"{\i}de associ\'ee. D'apr\`es le lemme qui pr\'ec\`ede, il existe un polyn\^ome~$P$ \`a coefficients dans~$\Hs(y)$ et un nombre r\'eel~\mbox{$s>0$} tels que le domaine affino\"{\i}de de~$\pi^{-1}(y)$ d\'efini par $\{|P|\le s\}$ ait un bord de Shilov r\'eduit au point~$x$. 

Chacun des coefficients du polyn\^ome~$P$ est limite de quotients d'\'el\'ements de $k\{T_{1},\ldots,T_{p}\}$ et m\^eme de $k[T_{1},\ldots,T_{p}]$. Il existe donc un polyn\^ome~$Q$ \`a coefficients dans~$k[T_{1},\ldots,T_{p}]$ et un \'el\'ement~$q$ de~$k[T_{1},\ldots,T_{p}]$, non nul en~$y$, tels que
\[\pi^{-1}(y)\cap \{|P|\le s\} = \pi^{-1}(y) \cap \{|Q|\le |q|s\} = \pi^{-1}(y) \cap \{|Q|\le |q(y)|s\}.\]

D\'efinissons un domaine affino\"{\i}de de~$\D^{p+1}_{k}(\br)$ par
\[V = \pi^{-1}(W) \cap \{|Q|\le |q(y)|s\}.\]
D'apr\`es le lemme~\ref{lem:Shilov}, le bord de Shilov de cet affino\"{\i}de est r\'eduit au point~$x$.

L'assertion finale de l'\'enonc\'e est claire puisque, \`a chaque \'etape, on ne fait intervenir dans la construction qu'un nombre fini d'\'el\'ements de~$k$, correspondant aux coefficients du polyn\^ome~$Q$.
\end{proof}

\begin{rem}
Les deux faits suivants d\'ecoulent de la preuve.
\begin{enumerate}[\it i)]
\item Le domaine de Weierstra\ss~$V$ peut \^etre d\'efini par~$p$ polyn\^omes.
\item Dans le cas o\`u le point~$x$ satisfait $s(x)=p$, le domaine de Weierstra\ss~$V$ peut \^etre choisi strictement affino\"{\i}de.
\end{enumerate}
\end{rem}

\begin{cor}\label{cor:staffs}
Soit~$X$ un espace strictement $k$-affino\"{\i}de irr\'eductible de dimension~$d$. Soit~$x$ un point de~$X$ tel que $s(x)=d$. Alors il existe un domaine rationnel strictement affino\"{\i}de de~$X$ dont le bord de Shilov est r\'eduit au point~$x$.  
\end{cor}
\begin{proof}
Le th\'eor\`eme de normalisation de Noether assure qu'il existe un morphisme fini $\varphi : X \to \D^d_{k}$ (o\`u nous avons not\'e~$\D^d_{k}$ le disque $\D^d_{k}(1,\dots,1)$). Nous avons $s(\varphi(x))=s(x)=d$. Par cons\'equent, il existe un domaine  de {Weierstra\ss} strictement affino\"{\i}de~$V_{0}$ de~$\D^d_{k}$ dont le bord de Shilov est r\'eduit au point~$\varphi(x)$. Son image r\'eciproque $V=\varphi^{-1}(V_{0})$ est un domaine de {Weierstra\ss} strictement affino\"{\i}de de~$X$.

Un argument de r\'eduction similaire \`a celui utilis\'e dans la preuve du lemme~\ref{lem:Shilov} montre que le bord de Shilov de~$V$ est \'egal \`a $\varphi^{-1}(\varphi(x))$. Notons~$\As$ l'alg\`ebre de~$X$ et~$\As_{V}$ celle de~$V$. Consid\'erons un \'el\'ement~$f$ de~$\As_{V}^\circ$ dont la r\'eduction~$\ti{f}$ s'annule sur toutes les composantes irr\'eductibles de~$\ti{V}$ sauf celle contenant~$\ti{x}$. Puisque~$V$ est un domaine de {Weierstra\ss} de~$X$, $\As$~est dense dans~$\As_{V}$ et nous pouvons supposer que~$f\in\As$. D'apr\`es~\cite{BGR}, proposition~7.2.6/3 (et~\cite{TemkinII}, proposition~3.1 dans le cas de valuation triviale), la r\'eduction du domaine affino\"{\i}de~$W$ de~$V$ d\'efini par $\{|f|=1\}$ est isomorphe \`a~$D(\ti{f}) \subset \ti{V}$. Son bord de Shilov est donc le singleton~$\{x\}$.
\end{proof}

\begin{cor}\label{cor:staffst}
Soit~$X$ un espace strictement $k$-affino\"{\i}de irr\'eductible de dimension~$d$. Soit~$x$ un point d'Abhyankar de~$X$. Alors il existe un domaine rationnel de~$X$ dont le bord de Shilov est r\'eduit au point~$x$.  
\end{cor}
\begin{proof}
Construisons un morphisme $\varphi : X \to \D^d_{k}$ et des domaines affino\"{\i}des~$V$ et~$V_{0}$ comme pr\'ec\'edemment. Soit~$\br$ un polyrayon $k$-libre tel que $V_{0}\ho_{k} k_{\br}$ soit strictement $k_{\br}$-affino\"{\i}de. \`A tout point~$y$ de~$V$ ou~$V_{0}$ on peut associer un point~$y_{\br}$ de~$V\ho_{k} k_{\br}$ ou~$V_{0}\ho_{k} k_{\br}$, par la construction d\'ecrite au lemme~\ref{lem:extkr}. On v\'erifie que les points du bord de Shilov de $V\ho_{k} k_{\br}$ (resp. $V_{0}\ho_{k} k_{\br}$) sont exactement ceux de la forme~$\gamma_{r}$, o\`u~$\gamma$ est un point du bord de Shilov de~$V$ (resp.~$V_{0}$). Nous pouvons alors reprendre la preuve pr\'ec\'edente pour montrer que le bord de Shilov de~$V$ est \'egal \`a $\varphi^{-1}(\varphi(x))$. La fin de la preuve est identique au cas strictement affino\"{\i}de.
\end{proof}

\begin{rem}
Dans les corollaires~\ref{cor:staffs} et~\ref{cor:staffst}, le domaine rationnel peut \^etre d\'efini par~$d+1$ \'el\'ements.
\end{rem}

\begin{cor}
Soit~$X$ un espace $k$-affino\"{\i}de. Soit~$x$ un point d'Abhyankar de~$X$ appartenant \`a une seule composante irr\'eductible~$C$ de~$X$ et supposons que cette composante soit strictement $k$-affino\"{\i}de. Alors il existe un domaine rationnel de~$X$ dont le bord de Shilov est r\'eduit au point~$x$. 

En outre, si~$s(x)=\dim_{k,x}(X)$, le domaine rationnel peut \^etre choisi strictement affino\"{\i}de.
\end{cor}
\begin{proof}
D'apr\`es le corollaire~\ref{cor:staffst}, il existe un entier~$p$, des \'el\'ements~$f_{1},\dots,f_{p}$ de l'alg\`ebre~$\As_{C}$ de~$C$ et des nombres r\'eels $s_{1},\dots,s_{p},t_{1},\ldots,t_{p}$ tels que le point~$x$ soit l'unique point du bord de Shilov du domaine affino\"{\i}de~$V$ de~$C$ d\'efini par $V = \bigcap_{1\le i\le p} \{y\in C\, |\, s_{i} \le |f_{i}(y)|\le t_{i}\}$. Relevons les~$f_{i}$ en des \'el\'ements~$g_{i}$ de l'alg\`ebre~$\As$ de~$X$.

Notons~$D$ la r\'eunion des composantes irr\'eductibles de~$X$ ne contenant pas~$x$. Soit~$f$ un \'el\'ement de~$\As$ qui est nul sur~$D$ mais ne s'annule pas en~$x$. Consid\'erons alors
\[\renewcommand{\arraystretch}{1.5}{\begin{array}{rcl}
W &=&\disp \Big(\bigcap_{1\le i\le p} \{y\in X\, |\, s_{i} \le |g_{i}(y)|\le t_{i}\}\Big) \cap \{y\in X\, |\, |f(y)|=|f(x)|\}\\
&=& V \cap \{y\in X\, |\, |f(y)|=\|f\|_{V}\}.
\end{array}}\]
C'est un domaine rationnel de~$X$ dont le bord de Shilov est r\'eduit au point~$x$ d'apr\`es~\cite{TemkinII}, proposition~3.1. 

On d\'emontre de m\^eme la seconde partie du r\'esultat.
\end{proof}

\begin{rem}
Dans ce corollaire, le domaine rationnel peut \^etre d\'efini par~$\dim_{x}(X)+2$ \'el\'ements.
\end{rem}

\begin{rem}
Soit~$X$ un espace affino\"{\i}de. Consid\'erons un domaine rationnel $V = \bigcap_{1\le i\le p} \{y\in X\, |\, s_{i} \le |f_{i}(y)|\le t_{i}\}$ de~$X$ dont le bord de Shilov est un singleton~$\{x\}$. D'apr\`es~\cite{TemkinII}, proposition~3.1, le bord de Shilov du domaine rationnel $\bigcap_{1\le i\le p} \{y\in X\, |\,  |f_{i}(y)| = |f_{i}(x)|=\|f\|_{V}\}$ est encore le singleton~$\{x\}$.

Jointe au corollaire qui pr\'ec\`ede, cette remarque permet de retrouver le r\'esultat d'un th\'eor\`eme de T.~de Fernex, L.~Ein et~S.~Ishii qui assure que toute valuation divisorielle sur une vari\'et\'e complexe affine peut \^etre d\'etermin\'ee par sa valeur sur un nombre fini de fonctions (\emph{cf.}~\cite{valuationsarcs}, theorem~0.2 pour un \'enonc\'e pr\'ecis).
\end{rem}

\medskip

Revenons maintenant aux points d'Abhyankar des disques.

\begin{cor}\label{cor:exttfs+tmax}
Soient~$p\in\N$ et~$\br \in (\R_{+}^*)^p$. Soit~$x$ un point d'Abhyankar de~$\D^p_{k}(\br)$. Il existe un sous-corps~$\ell$ de~$k$ de type d\'enombrable sur~$k_{p}$ v\'erifiant la propri\'et\'e suivante : pour tout corps~$\ell'$ tel que $\ell \subset \ell' \subset k$, si $\pi : \D^p_{k}(\br) \to \D^p_{\ell'}(\br)$ d\'esigne le morphisme de changement de base, alors~$x$ est l'unique point du bord de Shilov de la fibre~$\pi^{-1}(\pi(x))$. 
\end{cor}
\begin{proof}
Consid\'erons le domaine affino\"{\i}de~$V$ dont il est question dans la proposition~\ref{prop:s+tmax} et~$\ell$ un sous-corps de~$k$ de type d\'enombrable sur~$k_{p}$ sur lequel il est d\'efini. Il suffit de d\'emontrer la propri\'et\'e pour le corps~$\ell$. Il existe un domaine affino\"{\i}de~$W$ de~$\D^p_{\ell}(\br)$ tel que $\pi^{-1}(W) = V$. Le point~$\pi(x)$ appartient \`a~$W$ et tous les points de~$\pi^{-1}(\pi(x))$ appartiennent donc \`a~$V$. Par cons\'equent, pour tout $y\in\pi^{-1}(\pi(x))$, nous avons
\[\forall P \in k[T_{1},\ldots,T_{p}], |P(y)| \le \|P\|_{V} = |P(x)|.\]
Le point~$x$ est donc l'unique point du bord de Shilov de cette fibre. 
\end{proof}

Nous allons maintenant d\'emontrer un r\'esultat analogue valant pour tous les points des disques, par une sorte de passage \`a la limite.

\begin{thm}\label{thm:exttf}
Soient~$p\in\N$ et~$\br \in (\R_{+}^*)^p$. Pour tout point~$x$ de~$\D^p_{k}(\br)$, il existe un sous-corps~$\ell$ de~$k$ de type d\'enombrable sur~$k_{p}$ v\'erifiant la propri\'et\'e suivante : pour tout corps~$\ell'$ tel que $\ell \subset \ell' \subset k$, si $\pi : \D^p_{k}(\br) \to \D^p_{\ell'}(\br)$ d\'esigne le morphisme de changement de base, alors~$x$ est l'unique point du bord de Shilov de la fibre~$\pi^{-1}(\pi(x))$. 
\end{thm}
\begin{proof}
Soit~$x$ un point de~$\D^p_{k}(\br)$. Nous allons d\'emontrer que~$x$ v\'erifie la seconde propri\'et\'e par r\'ecurrence sur la quantit\'e~$p-s(x)-t(x)$. Si elle est nulle, le point~$x$ est un point d'Abhyankar et le corollaire~\ref{cor:exttfs+tmax} permet de conclure.

Soit~$m\in\N$ et supposons avoir d\'emontr\'e le r\'esultat lorsque~$p-s(x)-t(x)=m$. Soient~$p\in\N$, $\br = (r_{1},\ldots,r_{p}) \in (\R_{+}^*)^p$ et~$x$ un point de~$\D^p_{k}(\br)$ tel que $p-s(x)-t(x)=m+1$. Posons~$\br'= (r_{1},\ldots,r_{p-1})$ et notons $\varphi : \D^p_{k}(\br) \to \D^{p-1}_{k}(\br')$ le morphisme de projection sur les~$p-1$ premi\`eres coordonn\'ees. Posons $y=\varphi(x)$. Puisque $s(x)+t(x) < p$, quitte \`a r\'eordonner les variables, nous pouvons donc supposer que $s(y)+t(y) = s(x)+t(x)$.

Pla\c{c}ons-nous un moment dans la fibre $\varphi^{-1}(y) \simeq \D^1_{\Hs(y)}(r_{p})$. Il existe une suite de polyn\^omes~$(P_{n}(T_{p}))_{n\ge 0}$ \`a coefficients dans~$\Hs(y)$ et une suite de nombres r\'eels strictement positifs~$(r_{n})_{n\ge 0}$ v\'erifiant les propri\'et\'es suivantes :
\begin{enumerate}
\item pour tout~$n$, le domaine affino\"{\i}de~$V_{n}$ de~$\D^1_{\Hs(y)}(r_{p})$ d\'efini par~$\{|P_{n}|\le r_{n}\}$ contient~$x$\ ;
\item pour tout~$n$, le bord de Shilov de~$V_{n}$ est r\'eduit \`a un point, que nous noterons~$\gamma_{n}$\ ;
\item la suite d'affino\"{\i}des~$(V_{n})_{n\ge 0}$ est d\'ecroissante\ ;
\item la suite~$(\gamma_{n})_{n\ge 0}$ tend vers~$x$.
\end{enumerate}

Les \'el\'ements de~$\Hs(y)$ \'etant limites de quotients d'\'el\'ements de $k[T_{1},\ldots,T_{p-1}]$, en proc\'edant comme dans la preuve de la proposition~\ref{prop:s+tmax}, nous pouvons supposer que les coefficients des polyn\^omes~$P_{n}(T_{p})$ appartiennent \`a $k[T_{1},\ldots,T_{p-1}]$. Il existe donc un sous-corps~$\ell_{0}$ de~$k$ de type d\'enombrable sur~$k_{p}$ tel que tous les polyn\^omes~$P_{n}$ appartiennent \`a $\ell_{0}[T_{1},\ldots,T_{p}]$. En outre, nous avons $(p-1)-s(y)-t(y)=m$. Par hypoth\`ese de r\'ecurrence, il existe un sous-corps~$\ell$ de~$k$ de type d\'enombrable sur~$k_{p}$ tel que si $\pi' : \D^{p-1}_{k}(\br') \to \D^{p-1}_{\ell}(\br')$ d\'esigne le morphisme de changement de base, alors~$y$ est l'unique point du bord de Shilov de la fibre~$\pi'^{-1}(\pi'(y))$. Nous pouvons supposer que~$\ell$ contient~$\ell_{0}$. Quitte \`a le remplacer par le compl\'et\'e d'une cl\^oture alg\'ebriquement, nous pouvons \'egalement supposer qu'il est alg\'ebriquement clos.

Remarquons qu'il suffit de d\'emontrer le r\'esultat pour le corps~$\ell'=\ell$. Notons $\pi : \D^{p}_{k}(\br) \to \D^{p}_{\ell}(\br)$ le morphisme de changement de base. Puisque~$\ell$ est alg\'ebriquement clos, d'apr\`es le corollaire~\ref{cor:cime}, le point~$\pi(x)$ est universel et le bord de Shilov de la fibre~$\pi^{-1}(\pi(x))$ poss\`ede un unique point, que nous noterons~$z$. Tout \'el\'ement de $k[T_{1},\ldots,T_{p-1}]$ atteint son maximum sur~$\pi^{-1}(\pi(x))$ en~$z$. Par cons\'equent, $\varphi(z)=y$. Pour tout~$n$, il existe un domaine affino\"{\i}de~$W_{n}$ de~$\D^1_{\Hs(\pi'(y))}(r_{p})$ tel que $\pi^{-1}(W_{n}) = V_{n}$. Pour tout~$n$, le point~$\pi(x)$ appartient \`a~$W_{n}$ et le point~$z$ appartient donc \`a~$V_{n}$. Par cons\'equent, pour tout $P \in k[T_{1},\ldots,T_{p}]$, nous avons
\[|P(z)| \le \inf_{n\ge 0} (\|P\|_{V_{n}\cap \varphi^{-1}(y)}) = \inf_{n\ge 0} (|P(\gamma_{n})|) = |P(x)|,\]
d'o\`u l'on d\'eduit l'\'egalit\'e des points~$z$ et~$x$.
\end{proof}

\begin{rem}
On peut adapter le raisonnement pour montrer que les points d'Abhyankar sont denses dans~$\D^p_{k}(\br)$ et obtenir ainsi une preuve \'el\'ementaire du corollaire~\ref{cor:Abhyankardense} dans le cas d'un polydisque, puis dans le cas g\'en\'eral par extension des scalaires et normalisation de Noether.
\end{rem}

\section{Applications \`a la topologie des espaces de Berkovich}\label{section:topo}

Nous allons maintenant tirer quelques cons\'equences topologiques des r\'esultats que nous avons d\'emontr\'e, ainsi que nous l'avons annonc\'e dans l'introduction. Le r\'esultat dont d\'ecouleront tous les autres est celui qui assure qu'un point adh\'erent \`a une partie est limite d'une suite de points de cette partie. Les espaces topologiques v\'erifiant cette propri\'et\'e portent un nom.

\begin{defi}
Un espace topologique~$X$ est dit {\bf de Fr\'echet-Urysohn} si pour toute partie~$A$ de~$X$, tout point de l'adh\'erence~$\bar{A}$ de~$A$ est limite d'une suite de points de~$A$.
\end{defi}

Dans un premier temps, nous nous int\'eresserons aux disques d\'efinis sur un corps alg\'ebriquement clos, afin de pouvoir utiliser directement les r\'esultats de la section~\ref{section:universeldisque}.

\begin{prop}\label{prop:cascime}
Soient~$p\in\N$ et~$\br \in (\R_{+}^*)^p$. Supposons que le corps~$k$ est alg\'ebriquement clos. Alors l'espace~$\D^p_{k}(\br)$ est un espace de Fr\'echet-Urysohn.
\end{prop}
\begin{proof}
Soient~$A$ une partie de~$\D^p_{k}(\br)$, que nous pouvons supposer non vide, et~$x$ un point adh\'erent \`a~$A$. Nous voulons montrer que le point~$x$ est limite d'une suite de points de~$A$. Con\-si\-d\'e\-rons le sous-corps~$\ell$ de~$k$ associ\'e \`a ce point par le th\'eor\`eme~\ref{thm:exttf}. C'est une extension de type d\'enombrable fini du sous-corps premier~$k_{p}$ de~$k$. 

Pour tout entier $n\ge 0$, nous allons construire par r\'ecurrence un corps~$\ell_{n}$, une suite d\'ecroissante~$(V_{n}^{(m)})_{m\ge 0}$ de parties de~$\D^p_{\ell_{n}}(\br)$ et des points~$x_{n}$ et~$y_{n}$ de~$\D^p_{\ell_{n}}(\br)$ qui v\'erifient les propri\'et\'es suivantes : pour tout $n\ge 0$,
\begin{enumerate}[\it i)]
\item le corps~$\ell_{n}$ est un sous-corps alg\'ebriquement clos de~$k$ de type d\'enombrable sur~$\ell$ (et donc sur~$k_{p}$) ;
\item le point~$x_{n}$ est le projet\'e du point~$x$ sur~$\D^p_{\ell_{n}}$ ;
\item la suite~$(V_{n}^{(m)})_{m\ge 0}$ forme un syst\`eme fondamental de voisinages du point~$x_{n}$ ;
\item le point~$y_{n}$ est rationnel sur~$\ell_{n}$, appartient \`a~$V_{n}^{(n)}$ et est le projet\'e d'un point~$y'_{n}$ de~$A$ ;
\end{enumerate}
et, pour tous $n'\ge n \ge 0$,
\begin{enumerate}
\item[\it v)] $\ell_{n} \subset \ell_{n'}$ (nous noterons $\pi_{n',n} : \D^p_{\ell_{n'}} \to \D^p_{\ell_{n}}$ la projection associ\'ee) ;
\item[\it vi)] pour tout $m\ge 0$, $V_{n'}^{(m)} \subset \pi_{n',n}^{-1}(V_{n}^{(m)})$.
\end{enumerate}

Initialisons la r\'ecurrence. Pour cela, on choisit un point~$y'_{0}$ de~$A$. Il existe un sous-corps~$\ell_{0}$ de~$k$ de type d\'enombrable sur~$\ell$ qui contient le corps associ\'e au point~$y'_{0}$ par le th\'eor\`eme~\ref{thm:exttf}. Notons~$y_{0}$ la projection de ce point sur~$\D^p_{\ell_{0}}$. Quitte \`a remplacer~$\ell_{0}$ par le compl\'et\'e de sa fermeture alg\'ebrique dans~$k$, nous pouvons supposer que c'est un corps alg\'ebriquement clos. Posons $V_{0}^{(0)} = \D^p_{\ell_{0}}(\br)$ et notons~$x_{0}$ le projet\'e du point~$x$ sur~$\D^p_{\ell_{0}}(\br)$. On choisit ensuite une suite d\'ecroissante $(V_{0}^{(m)})_{m\ge 1}$ de parties de~$\D^p_{\ell_{0}}(\br)$ qui forme un syst\`eme fondamental de voisinages du point~$x_{0}$.

Soit~$n\ge 0$ et supposons avoir construit les objets du rang~$0$ au rang~$n$ de sorte qu'ils v\'erifient les propri\'et\'es demand\'ees. La partie~$V_{n}^{(n+1)}$ de $\D^p_{\ell_{n}}$ est un voisinage de~$x_{n}$. Son image r\'eciproque dans $\D^p_{k}(\br)$ est un voisinage de~$x$ et elle contient donc un point~$y'_{n+1}$ de~$A$. Il existe un sous-corps~$\ell_{n+1}$ de~$k$ de type d\'enombrable sur~$\ell_{n}$ (et donc sur~$\ell$) qui contient le corps associ\'e au point~$y'_{n+1}$ par le th\'eor\`eme~\ref{thm:exttf}. Notons~$y_{n+1}$ la projection de ce point sur~$\D^p_{\ell_{n+1}}(\br)$. Quitte \`a remplacer~$\ell_{n+1}$ par le compl\'et\'e de sa fermeture alg\'ebrique dans~$k$, nous pouvons supposer que c'est un corps alg\'ebriquement clos. Pour $m\le n+1$, posons $V_{n+1}^{(m)} = \pi_{n+1,n}^{-1}(V_{n}^{(m)})$. Notons~$x_{n+1}$ le projet\'e du point~$x$ sur~$\D^p_{\ell_{n+1}}(\br)$. Nous choisissons ensuite une suite d\'ecroissante $(W_{n+1}^{(m)})_{m\ge n+2}$ de parties de~$\D^p_{\ell_{n+1}}(\br)$ qui forme un syst\`eme fondamental de voisinages du point~$x_{n+1}$. Finalement, pour $m\ge n+2$, nous posons $V_{n+1}^{(m)} = W_{n+1}^{(m)}\cap \pi_{n+1,n}^{-1}(V_{n}^{(m)})$.\\

Notons~$\ell'_{\infty}$ le sous-corps de~$k$ engendr\'e par les~$\ell_{n}$, pour~$n\ge 0$, et~$\ell_{\infty}$ son compl\'et\'e. Ces deux corps sont alg\'ebriquement clos. Notons~$x_{\infty}$ le projet\'e du point~$x$ sur~$\D^p_{\ell_{\infty}}(\br)$. Pour $n\ge 0$, nous noterons~$z_{n}$ le projet\'e du point~$y'_{n}$ sur~$\D^p_{\ell_{\infty}}(\br)$ et $\pi_{n} : \D^p_{\ell_{\infty}}(\br) \to \D^p_{\ell_{n}}(\br)$ le morphisme de changement de base. Soit~$U$ un voisinage de~$x_{\infty}$ dans~$\D^p_{\ell_{\infty}}(\br)$. Il existe un entier~$n'$ et un voisinage~$V$ de~$\pi_{n'}(x_{\infty})=x_{n'}$ dans~$\D^p_{\ell_{n'}}(\br)$ tel que $\pi_{n'}^{-1}(V)=U$. Il existe donc un entier $n''\ge n'$ tel que $V_{n'}^{(n'')} \subset V$. Par cons\'equent, quel que soit $n\ge n''$, nous avons $\pi_{n}^{-1}(V_{n}^{(n)}) \subset U$ et donc $z_{n} \in U$. La suite $(z_{n})_{n\ge 0}$ de~$\D^p_{\ell_{\infty}}(\br)$ converge donc vers~$x_{\infty}$.\\  

D'apr\`es le corollaire~\ref{cor:cime}, les points~$z_{n}$ et le point~$x_{\infty}$ sont universels sur~$\ell_{\infty}$. On d\'eduit du corollaire~\ref{cor:remonte} que la suite $(\sigma_{k}(z_{n}))_{n\ge 0}$ tend vers~$\sigma_{k}(x_{\infty})$. Or, pour tout~$n$, $\sigma_{k}(z_{n}) = y'_{n}$, car~$\ell_{\infty}$ contient~$\ell_{n}$ et $\sigma_{k}(x_{\infty})=x$, car~$\ell_{\infty}$ contient~$\ell$.
\end{proof}

\begin{thm}\label{thm:suite}
Tout espace $k$-analytique est un espace de Fr\'echet-Urysohn.
\end{thm}
\begin{proof}
Soient~$X$ un espace $k$-analytique et~$A$ une partie de~$X$. Nous souhaitons montrer que tout point adh\'erent \`a~$A$ est limite d'une suite de points de~$A$.

Soit~$K$ une extension alg\'ebriquement close du corps~$k$. Notons $\pi : X\ho_{k} K \to X$ le morphisme de changement de base. Il suffit de d\'emontrer le r\'esultat pour l'espace~$X\ho_{k} K$ et la partie~$\pi^{-1}(A)$. En effet, le morphisme~$\pi$, en tant que morphisme topologiquement propre entre espaces localement compacts, est ferm\'e. Tout point adh\'erent~$x$ \`a~$A$ poss\`ede donc une pr\'eimage~$y$ dans l'adh\'erence de~$\pi^{-1}(A)$. S'il existe une suite~$(y_{n})_{n\ge 0}$ de points de~$\pi^{-1}(A)$ qui tend vers~$y$, alors la suite~$(\pi(y_{n}))_{n\ge 0}$ tend vers~$x$. Nous pouvons donc supposer que le corps~$k$ est alg\'ebriquement clos.

Soit~$x$ un point adh\'erent \`a~$A$. Il poss\`ede un voisinage qui est r\'eunion finie de domaines affino\"{\i}des. Quitte \`a remplacer~$X$ par l'un de ces affino\"{\i}des et~$A$ par sa trace sur icelui, nous pouvons supposer que~$X$ est affino\"{\i}de. C'est donc un ferm\'e de Zariski d'un disque et nous pouvons finalement supposer que $X = \D^p_{k}(\br)$. Nous concluons alors par la proposition~\ref{prop:cascime}.
\end{proof}

Les espaces de Fr\'echet-Urysohn font partie des espaces dits s\'equentiels : les parties ouvertes et ferm\'ees peuvent y \^etre caract\'eris\'ees par des suites, dans le sens que nous pr\'ecisons ci-dessous.

\begin{defi}
Soit~$X$ un espace topologique. 
\begin{enumerate}[\it i)]
\item Une partie~$A$ de~$X$ est dite {\bf s\'equentiellement ouverte} si tout suite d'\'el\'e\-ments de~$X$ qui converge vers un point de~$A$ ne poss\`ede qu'un nombre fini de termes hors de~$A$.
\item Une partie~$A$ de~$X$ est dite {\bf s\'equentiellement ferm\'ee} si tout point de~$X$ qui est limite d'une suite d'\'el\'ements de~$A$ appartient \`a~$A$.
\end{enumerate}
\end{defi}

\begin{cor}
Soient~$X$ un espace $k$-analytique et~$A$ une partie de~$X$.
\begin{enumerate}[\it i)]
\item La partie~$A$ est ouverte si, et seulement si, elle est s\'equentiellement ouverte.
\item La partie~$A$ est ferm\'ee si, et seulement si, elle est s\'equentiellement ferm\'ee.
\end{enumerate}
\end{cor}

Pour d'autres r\'esultats sur les espaces de Fr\'echet-Urysohn et les espaces s\'equentiels, nous renvoyons \`a l'article~\cite{sequences}. S.~P.~Franklin y d\'emontre notamment que tout espace s\'equentiel est quotient d'un espace m\'etrique. Ce r\'esultat vaut donc en particulier pour les espaces $k$-analytiques.

En combinant le th\'eor\`eme~\ref{thm:suite} et le corollaire~\ref{cor:Abhyankardense}, nous obtenons le r\'esultat suivant.

\begin{cor}
L'ensemble des points d'Abhyankar d'un espace analytique est s\'equentiellement dense.
\end{cor}

Pour les espaces strictement $k$-analytiques, nous pouvons obtenir un r\'esultat plus pr\'ecis.

\begin{cor}
Supposons que la valuation de~$k$ n'est pas triviale. Posons $c=\dim_{\Q}(\R_{+}^*/\sqrt{|k^*|})$. Soit~$X$ un espace strictement $k$-affino\"{\i}de et notons~$d$ le minimum des dimensions de ses composantes irr\'eductibles. Soient~$s,t\in\N$ tels que $s+t\le d$ et $t\le c$. Alors l'ensemble des points~$x$ de~$X$ v\'erifiant $s(x)=s$ et $t(x)=t$ est dense et s\'equentiellement dense dans~$X$.
\end{cor}
\begin{proof}
D'apr\`es le th\'eor\`eme~\ref{thm:suite}, il suffit de montrer que l'ensemble indiqu\'e est dense.

Soient~$x$ un point de~$X$ et~$U$ un voisinage strictement $k$-affino\"{\i}de de~$x$. Sa dimension~$d'$ est n\'ecessairement sup\'erieure \`a~$d$. D'apr\`es le th\'eor\`eme de normalisation de Noether, il existe un morphisme fini et surjectif $\varphi : U \to \D^{d'}_{k}=\D^{d'}_{k}(1,\dots,1)$. Les conditions impos\'ees \`a~$s$ et~$t$ assurent que le disque~$\D^{d'}_{k}$ contient un point~$y$ tel que $s(y)=s$ et $t(y)=t$. Puisque le morphisme~$\varphi$ est fini, tout ant\'ec\'edent~$x$ de~$y$ satisfait les m\^emes \'egalit\'es.
\end{proof}

\begin{rem}
On retrouve ainsi un r\'esultat de C.~Favre (\emph{cf.}~\cite{Countability}, corollary~C) assurant la densit\'e et la densit\'e s\'equentielle des points divisoriels dans les espaces analytiques sur~$k(\!(T)\!)$.
\end{rem}

\'Enon\c{c}ons un autre corollaire frappant. Signalons qu'il est utilis\'e de fa\c{c}on essentielle dans la preuve de l'analogue ultram\'etrique du th\'eor\`eme de Montel par C.~Favre, J.~Kiwi et E.~Trucco (\emph{cf.}~\cite{Montel}).

\begin{cor}
Soit~$X$ un espace $k$-analytique. Toute partie compacte de~$X$ est s\'equentiellement compacte.
\end{cor}
\begin{proof}
Soit~$A$ une partie compacte de~$X$. Soit $(a_{n})_{n\in\N}$ une suite de points de~$A$. L'ensemble des valeurs d'adh\'erence de la suite $(a_{n})_{n\in\N}$ est un compact non vide et on conclut par le th\'eor\`eme~\ref{thm:suite}.
\end{proof}

Pour finir, nous allons montrer que les espaces $k$-analytiques satisfont une propri\'et\'e suppl\'ementaire, celle d'\^etre ang\'eliques. Rappelons tout d'abord qu'un espace topologique est dit $\omega$-compact lorsqu'il est s\'epar\'e et que toute suite poss\`ede une valeur d'adh\'erence. Cette derni\`ere condition \'equivaut \`a demander que de tout recouvrement ouvert d\'enombrable, on puisse extraire un recouvrement fini. En analyse fonctionnelle, le th\'eor\`eme d'Eberlein-\v{S}mulian assure que les notions de compacit\'e, compacit\'e s\'equentielle et $\omega$-compacit\'e sont \'equivalentes pour les parties d'un espace de Banach r\'eel muni de la topologie faible. La propri\'et\'e d'\^etre ang\'elique g\'en\'eralise les conclusions de ce th\'eor\`eme et A.~Grothendieck a notamment montr\'e dans~\cite{Grothendieck52}, que l'espace~$\Cs(K)$ muni de la convergence ponctuelle est ang\'elique, lorsque~$K$ est un espace $\omega$-compact. Nous renvoyons \`a l'ouvrage~\cite{Floret} pour plus d'informations sur cette notion.

\begin{defi}
Un espace topologique~$X$ est dit {\bf ang\'elique} si, pour toute partie relativement $\omega$-compacte~$A$ de~$X$, nous avons 
\begin{enumerate}[\it i)]
\item $A$ est relativement compacte\ ;
\item tout point de l'adh\'erence~$\bar{A}$ de~$A$ est limite d'une suite de points de~$A$.
\end{enumerate}
\end{defi}

\begin{prop}
Soit~$X$ un espace $k$-analytique dont la topologie est s\'epar\'ee. Toute partie $\omega$-compacte de~$X$ est ferm\'ee.
\end{prop}
\begin{proof}
Soient~$A$ une partie $\omega$-compacte de~$X$ et~$x$ un point de~$\bar{A}$. D'apr\`es le th\'eor\`eme~\ref{thm:suite}, il existe une suite $(x_{n})_{n\ge 0}$ de~$A$ qui tend vers~$x$. Puisque~$A$ est $\omega$-compacte, le point~$x$, qui est l'unique valeur d'adh\'erence de la suite $(x_{n})_{n\ge 0}$, appartient \`a~$A$. On en d\'eduit que $A=\bar{A}$.
\end{proof}

\begin{cor}
Soit~$X$ un espace $k$-analytique. Soient~$A$ et~$B$ deux parties de~$X$ telles que~$A\subset B$. Supposons que~$A$ est $\omega$-compacte et~$B$ compacte. Alors~$A$ est compacte.
\end{cor}

\begin{cor}\label{cor:paracompact}
Soit~$X$ un espace $k$-analytique paracompact. Toute partie $\omega$-compacte de~$X$ est compacte.
\end{cor}
\begin{proof}
Soit~$A$ une partie $\omega$-compacte de~$X$. L'espace~$X$ \'etant localement connexe, ses composantes con\-nexes sont ouvertes. Puisque la partie~$A$ est $\omega$-compacte, elle ne peut couper qu'un nombre fini de ces composantes. Nous pouvons donc supposer que~$X$ poss\`ede un nombre fini de composantes connexes. 

Puisque~$X$ est localement compact et paracompact, d'apr\`es~\cite{BourbakiTopo14}, I, \S 10, th\'eor\`eme~5, il existe une suite $(K_{n})_{n\ge 0}$ de parties compactes de~$X$ dont la r\'eunion recouvre~$X$. Puisque~$X$ est localement compact, d'apr\`es~{\it ibid.}, I, \S 9, proposition~15, nous pouvons supposer que, pour tout~$n$, on ait $K_{n} \subset \overset{\circ}{K_{n+1}}$. Puisque~$A$ est $\omega$-compacte, il existe un entier~$N$ tel que $A \subset K_{N}$. On conclut alors par le th\'eor\`eme pr\'ec\'edent. 
\end{proof}

En combinant le th\'eor\`eme~\ref{thm:suite} et le corollaire~\ref{cor:paracompact}, nous obtenons le r\'esultat voulu.

\begin{thm}
Toute partie d'un espace $k$-analytique paracompact est ang\'elique.
\end{thm}

\begin{rem}
Pour un espace $k$-analytique dont la topologie est s\'epar\'ee, la paracompacit\'e est fr\'equemment v\'erifi\'ee en pratique : c'est notamment le cas des courbes analytiques ainsi que des analytifi\'ees de vari\'et\'es alg\'ebriques.
\end{rem}

\nocite{}
\bibliographystyle{alpha}
\bibliography{biblio}

\begin{thebibliography}{dFEI08}

\bibitem[AT51]{ArtinTate}
Emil Artin and John~T. Tate.
\newblock A note on finite ring extensions.
\newblock {\em J. Math. Soc. Japan}, 3:74--77, 1951.

\bibitem[Ber90]{rouge}
Vladimir~G. Berkovich.
\newblock {\em Spectral theory and analytic geometry over non-{A}rchimedean
  fields}, volume~33 of {\em Mathematical Surveys and Monographs}.
\newblock American Mathematical Society, Providence, RI, 1990.

\bibitem[Ber99]{smoothI}
Vladimir~G. Berkovich.
\newblock Smooth {$p$}-adic analytic spaces are locally contractible.
\newblock {\em Invent. Math.}, 137(1):1--84, 1999.

\bibitem[BGR84]{BGR}
Siegfried Bosch, Ulrich G{\"u}ntzer, and Reinhold Remmert.
\newblock {\em Non-{A}rchimedean analysis}, volume 261 of {\em Grundlehren der
  Mathematischen Wissenschaften}.
\newblock Springer-Verlag, Berlin, 1984.
\newblock A systematic approach to rigid analytic geometry.

\bibitem[Bos69]{Orthonormalbasen}
S.~Bosch.
\newblock {Orthonormalbasen in der nichtarchimedischen Funktionentheorie.}
\newblock {\em Manuscr. Math.}, 1:35--57, 1969.

\bibitem[Bou71]{BourbakiTopo14}
N.~Bourbaki.
\newblock {\em \'{E}l\'ements de math\'ematique. {T}opologie g\'en\'erale.
  {C}hapitres 1 \`a 4}.
\newblock Hermann, Paris, 1971.

\bibitem[BR10]{BR}
Matthew Baker and Robert Rumely.
\newblock {\em Potential theory and dynamics on the {B}erkovich projective
  line}, volume 159 of {\em Mathematical Surveys and Monographs}.
\newblock American Mathematical Society, Providence, RI, 2010.

\bibitem[dFEI08]{valuationsarcs}
Tommaso de~Fernex, Lawrence Ein, and Shihoko Ishii.
\newblock {Divisorial valuations via arcs.}
\newblock {\em Publ. Res. Inst. Math. Sci.}, 44(2):425--448, 2008.

\bibitem[Duc]{RSS}
Antoine Ducros.
\newblock La structure des courbes analytiques.
\newblock En pr\'eparation.

\bibitem[Duc07]{variationdimension}
Antoine Ducros.
\newblock Variation de la dimension relative en g\'eom\'etrie analytique
  {p}-adique.
\newblock {\em Compos. Math.}, 143(6):1511--1532, 2007.

\bibitem[Duc09]{excellence}
Antoine Ducros.
\newblock Les espaces de {B}erkovich sont excellents.
\newblock {\em Ann. Inst. Fourier (Grenoble)}, 59(4):1443--1552, 2009.

\bibitem[Fab11]{Berkovichramification}
Xander Faber.
\newblock Topology and {G}eometry of the {B}erkovich {R}amification {L}ocus for
  {R}ational {F}unctions.
\newblock Pr\'epublication
  \href{http://arxiv.org/abs/1102.1432}{arXiv:1102.1432}, 2011.

\bibitem[Fav11]{Countability}
Charles Favre.
\newblock Countability properties of some {B}erkovich spaces.
\newblock Pr\'epublication
  \href{http://arxiv.org/abs/1103.6233}{arXiv:1103.6233}, 2011.

\bibitem[FKT12]{Montel}
Charles Favre, Jan Kiwi, and Eugenio Trucco.
\newblock A non-archimedean {M}ontel's theorem.
\newblock {\em Compos. Math.}, 2012.
\newblock \`A para\^{\i}tre.

\bibitem[Flo80]{Floret}
Klaus Floret.
\newblock {\em Weakly compact sets}, volume 801 of {\em Lecture Notes in
  Mathematics}.
\newblock Springer, Berlin, 1980.
\newblock Lectures held at S.U.N.Y., Buffalo, in Spring 1978.

\bibitem[Fra65]{sequences}
S.~P. Franklin.
\newblock Spaces in which sequences suffice.
\newblock {\em Fund. Math.}, 57:107--115, 1965.

\bibitem[Gro52]{Grothendieck52}
A.~Grothendieck.
\newblock Crit\`eres de compacit\'e dans les espaces fonctionnels g\'en\'eraux.
\newblock {\em Amer. J. Math.}, 74:168--186, 1952.

\bibitem[HL10]{HL}
Ehud Hrushovski and Fran\c{c}ois Loeser.
\newblock Non-archimedean tame topology and stably dominated types.
\newblock Pr\'epublication
  \href{http://arxiv.org/abs/1009.0252}{arXiv:1009.0252}, 2010.

\bibitem[Tem04]{TemkinII}
Michael Temkin.
\newblock On local properties of non-{A}rchimedean analytic spaces. {II}.
\newblock {\em Israel J. Math.}, 140:1--27, 2004.

\bibitem[Tem10]{stablemodification}
Michael Temkin.
\newblock Stable modification of relative curves.
\newblock {\em J. Algebraic Geom.}, 19(4):603--677, 2010.

\end{thebibliography}

\end{document}